\documentclass[a4paper,11pt,twoside]{article}

\usepackage[utf8]{inputenc}
\usepackage[T1]{fontenc}
\usepackage[english]{babel}
\usepackage{charter}
\usepackage{indentfirst}
\usepackage{xcolor}
\usepackage{verbatim}

\usepackage{hyperref}

\usepackage{pifont}

\usepackage[top=3.cm, bottom=4.0cm, left=2.2cm, right=2.2cm]{geometry}
\usepackage{mathtools}
\usepackage{amsmath}
\usepackage{amsthm}	%package pour les styles des ?nonc?s
\usepackage{amsfonts}	%pourles polices mathematiques vraisemblablement
\usepackage{amssymb}	%pour les symboles vraisemblablement
\usepackage{enumitem}
\usepackage[nobysame
           ,alphabetic
           ,initials
           ]{amsrefs}

\usepackage{
            ulem		% use \sout{...}
           ,soul		%  use \st{..}
} 

\numberwithin{equation}{section}
\numberwithin{figure}{section}
 \usepackage[nodayofweek]{datetime}

\renewcommand*{\thefootnote}{\fnsymbol{footnote}}

\title{Operator-splitting schemes for degenerate, non-local, conservative-dissipative systems}
\author{
\normalsize Daniel Adams\textit{$^{a,}$}\footnote{D.A was supported by The Maxwell Institute Graduate School in Analysis and its Applications, a Centre for Doctoral Training funded by the UK Engineering and Physical Sciences Research Council (grant EP/L016508/01), the Scottish Funding Council, Heriot-Watt University and the University of Edinburgh. } \\
        \small  d.t.s.adams@sms.ed.ac.uk 
\and 
\normalsize Manh Hong Duong\textit{$^{b}$}\footnote{The research of MHD was supported by EPSRC Grants EP/W008041/1 and EP/V038516/1.}\\
        \small  h.duong@bham.ac.uk
\and 
\normalsize Gon\c calo dos Reis\textit{$^{c,d,}$}\footnote{G.d.R. acknowledges support from the \emph{Funda{\c c}$\tilde{\text{a}}$o para a Ci$\hat{e}$ncia e a Tecnologia} (Portuguese Foundation for Science and Technology) through the project UIDB/00297/2020 (Centro de Matem\'atica e Aplica\c c$\tilde{\text{o}}$es CMA/FCT/UNL).} \\
        \small  G.dosReis@ed.ac.uk}
\date{
    \footnotesize 
    $^{a}$~Maxwell Institute for Mathematical Sciences School of Mathematics University of Edinburgh Edinburgh UK EH9 3FD
    \\
    $^{b}$~School of Mathematics, University of Birmingham, Birmingham B15 2TT, UK
    \\
    $^{c}$~School of Mathematics, University of Edinburgh, The King's Buildings, Edinburgh, UK
    \\
    $^{d}$~Centro de Matem\'atica e Aplica\c c$\tilde{\text{o}}$es (CMA), FCT, UNL, Portugal
    \\[2ex]
    \vspace{-0.8cm}
}
\theoremstyle{definition}
\newtheorem{theorem}{Theorem}[section]
\newtheorem{lemma}[theorem]{Lemma}

\newtheorem{defn}[theorem]{Definition}

\newtheorem{rem}[theorem]{Remark}
\newtheorem{prop}[theorem]{Proposition}

\newtheorem{assumption}[theorem]{Assumption}

\newcommand{\bN}{\mathbb{N}}

\newcommand{\bR}{\mathbb{R}}

\newcommand{\cA}{\mathcal{A}}

\newcommand{\cF}{\mathcal{F}}

\newcommand{\cP}{\mathcal{P}}

\newcommand{\x}{\mathbf{x}}

\newcommand{\divv}{\text{div}}

\definecolor{darkgreen}{rgb}{0,0.35,0}

\begin{document}

\selectlanguage{english}

\maketitle
%\tableofcontents
%%%%%%%%%%%%%%%%%%%%%%%%%%%%%%%%%%%%%%%%%%%%%%%%%%%%%%%%%%%%%%%%%
%%%%%%%%%%%%%%%%%%%%%%%%%%%%%%%%%%%%%%%%%%%%%%%%%%%%%%%%%%%%%%%%%
%%%%%%%%%%%%%%%%%%%%%%%%%%%%%%%%%%%%%%%%%%%%%%%%%%%%%%%%%%%%%%%%%
%%%%%%%%%%%%%%%%%%%%%%%%%%%%%%%%%%%%%%%%%%%%%%%%%%%%%%%%%%%%%%%%%
%%%%%%%%%%%%%%%%%%%%%%%%%%%%%%%%%%%%%%%%%%%%%%%%%%%%%%%%%%%%%%%%%
\renewcommand*{\thefootnote}{\arabic{footnote}}

%%%%%%%%%%%%%%%%%%%%%%%%%%%%%%%%%
\begin{abstract} 
In this paper, we develop a natural operator-splitting variational scheme for a general class of non-local, degenerate conservative-dissipative evolutionary equations. The splitting-scheme consists of two phases: a conservative (transport) phase and a dissipative (diffusion) phase. The first phase is solved exactly using the method of characteristic and  DiPerna-Lions theory while the second phase is solved  approximately using a JKO-type variational scheme that minimizes an energy functional with respect to a certain Kantorovich optimal transport cost functional.  In addition, we also introduce an entropic-regularisation of the scheme. We prove the convergence of both schemes to a weak solution of the evolutionary equation. We illustrate the generality of our work by providing a number of examples, including the kinetic Fokker-Planck equation and the (regularized) Vlasov-Poisson-Fokker-Planck equation. 
\end{abstract}
%%%%%%%%%%%%%%%%%%%%%%%%%%%%%%%%%
\vspace{0.3cm}

\noindent
{\bf 2020 AMS subject classifications:}\\
%%% See here: https://mathscinet.ams.org/msc/msc2020.html
Primary: 35K15, 35K55, Secondary: 65K05, 90C25

%\vspace{0.3cm}
%
%\noindent{\bf JEL subject classifications:}\\
%??, ??, ??
%%G11, G12, G13, C73
%%%%%%%%%%%%%%%%%%%%%%%%%%%%%%%%%

%
%
%
%%%%%%%%%%%%%%%%%%%%%%%%%%%%%%%%%%%%%%%%%%%%%%%%%%%%%%%%%%%%%%
%%%%%%%%%%%%%%%%%%%%%%%%%%%%%%%%%%%%%%%%%%%%%%%%%%%%%%%%%%%%%%
%%%%%%%%%%%%%%%%%%%%%%%%%%%%%%%%%%%%%%%%%%%%%%%%%%%%%%%%%%%%%%
%%%% \BEGIN SECTION
%%%%%%%%%%%%%%%%%%%%%%%%%%%%%%%%%%%%%%%%%%%%%%%%%%%%%%%%%%%%%%
% \newpage 
\footnotesize
\setcounter{tocdepth}{2}
\tableofcontents
\normalsize
\footnotesize
% \newpage 

%%%%%%%%%%%%%%%%%%%%%%%%%%%%%%%%%%%%%%%%%%%%%%%%%%%%%%%%%%%%%
%%%% \END SECTION
%%%%%%%%%%%%%%%%%%%%%%%%%%%%%%%%%%%%%%%%%%%%%%%%%%%%%%%%%%%%%
%%%%%%%%%%%%%%%%%%%%%%%%%%%%%%%%%%%%%%%%%%%%%%%%%%%%%%%%%%%%%
%%%%%%%%%%%%%%%%%%%%%%%%%%%%%%%%%%%%%%%%%%%%%%%%%%%%%%%%%%%%%
%
%
%

%
%
%
%%%%%%%%%%%%%%%%%%%%%%%%%%%%%%%%%%%%%%%%%%%%%%%%%%%%%%%%%%%%%%%%%%%%	
%%%%%%%%%%%%%%%%%%%%%%%%%%%%%%%%%%%%%%%%%%%%%%%%%%%%%%%%%%%%%%%%%%%%
%%% \BEGIN SECTION
%%%%%%%%%%%%%%%%%%%%%%%%%%%%%%%%%%%%%%%%%%%%%%%%%%%%%%%%%%%%%%%%%%%%
%%%%%%%%%%%%%%%%%%%%%%%%%%%%%%%%%%%%%%%%%%%%
%%%%%%%%%%%%%%%%%%%%%%%%%%%%%%%%%%%%%%%%%%%%%%%%%%%%%%%%%%%%
%%%%%%%%%%%%%%%%%%%%%%%%%%%%%%%%%%%%%%%%%%%%%%%%%%%%%%%%%%%%%%%
%%% BEGIN SECTION
%%%%%%%%%%%%%%%%%%%%%%%%%%%

\section{Introduction}\label{section introduction}
\subsection{Dissipative systems and Wasserstein gradient flows}
\label{sec Wasserstein gradient flow}
In their seminal work \cite{jordan1998variational} Jordan, Otto and Kinderlehrer show that the linear Fokker-Planck Equation (FPE)
\begin{equation}
    \label{FPE}
    \partial_t \rho = \mathrm{div}(\rho\nabla f)+\Delta \rho, 
\end{equation}
which is the forward Kolmogorov equation of the overdamped Langevin dynamics
\begin{equation}
\label{overdamped SDE}
dX(t)=-\nabla f(X(t))\,dt+\sqrt{2}dW(t),
\end{equation}
can be viewed as a gradient flow of the free energy functional with respect to the Wasserstein metric. Thereby the solution of the FPE can be iteratively approximated by the following minimising movement (steepest descent) scheme: given a time-step $h>0$ and defining $\rho^0_h:=\rho_0$, then the solution $\rho^n_h$ at the $n$-th step, $n=1,..., \lfloor \frac{T}{h}\rfloor$, is determined as the unique minimiser of the following minimisation problem 
\begin{equation}
\label{eq discrete general grad flow}
    \min_{\rho} \Big\{\frac{1}{2h} W_2^2(\rho^{n-1}_h, \rho)+\mathcal{F}_{\textrm{fpe}}(\rho)\Big\} 
    \quad \textrm{with} \quad \mathcal{F}_{\textrm{fpe}}(\rho):=\int f\rho+\rho\log\rho,
\end{equation}
over the space of the probability measures with finite second moment. Since then, the theory of Wasserstein gradient flows on the space of probability measures has made enormous progress, spanning research activity in various branches of mathematics including partial differential equations, probability theory, optimal transportation as well as geometric analysis. It provides a unified framework and tools for studying well-posedness,  regularity,  stability  and  quantitative  convergence  to  equilibrium  of  dissipative  systems.  Over  the  last twenty years, many evolutionary PDEs for models in biology, chemistry, mechanics, and physics have been analysed via this framework. Examples include the Fokker-Planck equation, porous medium equations, thin-film equations, nonlinear aggregation-diffusion equations, interface evolutions, as well as pattern formation and evolution \cites{villani2021topics, ambrosio2008gradient,santambrogio2015optimal}. %,MR3625852 
More recently, the theory has been extended to other settings such as discrete spaces \cites{Maas2011,Esposito2021} %Mielke2013 ,Chow2012
and quantum evolutions \cites{Carlen2017} % Mittnenzweig2017
and made intimate connections to the large deviation theory of stochastic processes \cites{adams2011large,Duong2013,Mielke2014}.
\subsection{Conservative-dissipative systems} Many important evolutionary equations arising from biology and physics are not gradient flows, but contain both conservative and dissipative dynamics. A prototypical example is the (generalized\footnote{In the classical Kramers equation, $F(p)=\frac{p^2}{2}$.}) Kramers' (or kinetic Fokker-Planck) equation \cites{Kramers40,risken1989fokker},
\begin{equation}
\label{KRequation}
\partial_t \rho = \underbrace{\Big(- \mathrm{div}_q (\rho p) + \mathrm{div}_p (\rho \nabla_q V)\Big)}_{\text{conservative part}}+ \underbrace{\Big(\mathrm{div}_p (\rho\nabla_p F) + \Delta_p \rho\Big)}_{\text{dissipative part}},
\end{equation}
for a density $\rho$ depending on $t\in \bR_+,q,p\in \bR^d$. In the above equation, we use the notation $\mathrm{div}_q$ and similar to indicate that the differential operator acts only on one variable. The Kramers equation is the forward Kolmogorov equation of the underdamped Langevin dynamics
%\begin{subequations}
%\label{eq:SDE}
%\begin{align}
%dQ(t)&=P(t)\,dt\label{sto1},
%\\dP(t)&=-\nabla V(Q(t))dt -\gamma P(t)\,dt + \sqrt{2\gamma}\, %dW(t)\label{sto2}.
%\end{align}
%\end{subequations}
\begin{equation}
    \label{eq:SDE}
d\begin{pmatrix}
Q\\P
\end{pmatrix}=\underbrace{\begin{pmatrix}
P\\ -\nabla V(Q)
\end{pmatrix}\,dt}_{\text{conservative dynamics}}+\underbrace{\begin{pmatrix}
0\\-\nabla F(P)\,dt+\sqrt{2}\, dW_t
\end{pmatrix}}_{\text{dissipative dynamics}}
\end{equation}
The Langevin dynamics~\eqref{eq:SDE} describes the movement of a particle (with unity mass) at position $Q$ and with momentum $P$ under the influence of three forces: an external force field ($-\nabla V(Q)$), a (possibly non-linear) friction ($-\nabla F(P)$) and a stochastic noise ($\sqrt{2}dW_t$). The Kramers equation \eqref{KRequation} characterizes the time evolution of the probability of finding the particle at time~$t$ at position $q$ and with momentum $p$. Unlike the Fokker-Planck equation \eqref{FPE}, which is purely dissipative, the Kramers equation~\eqref{KRequation} is a mixture of both conservative and dissipative dynamics. The first part in~\eqref{eq:SDE} is a deterministic Hamiltonian system with Hamiltonian energy $H(q,p) = p^2/2 + V(q)$. The evolution of this part is reversible and conserves the Hamiltonian. Correspondingly, the first part of~\eqref{KRequation} is also reversible and conserves the expectation of $H$,
\[
\mathcal H(\rho) := \int_{\mathbb{R}^{2d}} \rho(q,p)H(q,p) \, dqdp.
\]
On the other hand, the second part of \eqref{eq:SDE} is an overdamped Langevin dynamics (cf. \eqref{overdamped SDE}), but only in the $p$-variable. The corresponding part in \eqref{KRequation} is precisely a Fokker-Planck equation in $p$-variable (cf. \eqref{FPE}), which is a Wasserstein gradient flow in the $p$-variable. Because of the mixture of both conservative and dissipative effects, the full Kramers equation~\eqref{KRequation} is not a gradient flow, and the theory of Wasserstein gradient flows, in particular the JKO-minimizing movement scheme \eqref{eq discrete general grad flow}, is not directly applicable.

It is desirable to develop operator-splitting methods for conservative-dissipative systems that reflect the same division between conservative and dissipative effects. Developing structure-preserving schemes is currently of great interest both theoretically and computationally, according to
\cite{ottinger2018generic} ``an important challenge for the future is how the structure of thermodynamically admissible
evolution equations can be preserved under time-discretization, which is a key to successful numerical calculations''. For the Kramers equation, such an operator-splitting scheme is introduced in \cite{duong2014conservative}. However, the scheme developed in that work uses a complicated optimal transport cost functional for the dissipative part which does not capture the fact that it is simply a Wasserstein gradient flow in the momentum variable. More recently in \cites{carlier2017splitting} the authors introduce an operator-splitting scheme for a non-degenerate conservative-dissipative non-local-nonlinear diffusion equation
$$
\partial_t \rho + \divv \big( \rho b[\rho]\big)  =  \Delta P(\rho) + \divv \big(\rho\nabla f \big),
$$
where $b[\rho]$ is a divergence-free vector field for each $\rho$, and $P$ is the non-linear pressure function. The above equation does not cover the Kramers equation since the latter is a degenerate diffusion, in which the Laplacian only acts on the momentum variables. The degeneracy of the Kramers equation can also be seen in the Langevin dynamics \eqref{eq:SDE} where the noise is present only in the momentum but not on the position variables. A natural question arises
\begin{center}
\textit{Can we develop operator-splitting schemes for non-local, degenerate conservative-dissipative systems?}
\end{center}
In this paper, we address this question by providing a simple operator-splitting scheme for a general class of non-local, degenerate, conservative-dissipative systems. 

\subsection{Our contribution} 
We consider a general class of degenerate, non-local, conservative-dissipative evolutionary equations of the form 
\begin{equation}\label{eq main PDE}
    \partial_t \rho + \divv \big( \rho b[\rho]\big)  =   \divv \big(A\big(\nabla \rho+\rho\nabla f \big)\big),~~~~\rho(0,\cdot)=\rho_0(\cdot),
\end{equation}
where the unknown $\rho$ is a time dependent probability distribution on $[0,T]\times \bR^d$, $A\in \bR^{d\times d}$ is a semi-positive definite (symmetric) matrix (possibly degenerate), $f:\bR^d\to \bR$ a given energy potential,  $b:\cP(\bR^d) \times \bR^d\to \bR^d$ a divergence free  non-local vector field, and the probability density $\rho_0\in \cP_2^r(\bR^d)$ is the initial condition. Here, and throughout we denote $\cP(\bR^d)$ the space of Borel probability measures on $\bR^d$, $\cP_2(\bR^d)$ those with finite second moment, and $\cP_2^r(\bR^d)$ as those which are additionally absolutely continuous (with respect to the Lebesgue measure). Equation \eqref{eq main PDE} can be viewed as the forward Kolmogorov equation describing the time evolution of the distribution $\rho$ associated to the stochastic process $X$ satisfying the following SDE of McKean type
\begin{align}\label{eq SDE of McKean type}
   &dX(t)=b[\rho_t](X(t))dt- A \nabla f(X(t)) dt+\sqrt{2\sigma} dW(t),\qquad \rho_t=\text{Law}(X_t),
%   \nonumber
\end{align}
for a constant diffusion matrix $\sigma$, with $\sigma \sigma^T = A$. This serves as a general model for the dynamic limit of interacting particles, evolving under the influence of (weakly, self-stabilising) interaction force $b[\rho]$ depending on the own law of the process, and a potential drift $\nabla f$, whilst being perturbed by noise $W(t)$. Like the Kramers equation, \eqref{eq main PDE} contains both conservative and dissipative effects. The conservative part is represented via the divergence-free vector field (the transport part in the left-hand side of \eqref{eq main PDE}), in particular implying that the entropy will be preserved under this part. On the other hand, the dissipative part is given by the right hand side of \eqref{eq main PDE}, which resembles a $A$-Wasserstein gradient flow \cite{lisini2009nonlinear} (but note that $A$ can be degenerate). The aim of this paper is to develop operator-splitting schemes, which capture the conservative-dissipative splitting and take into account the degeneracy of the diffusion matrix, for solving \eqref{eq main PDE}.

Our operator-splitting scheme can be summarised as follows (details follow in Section \ref{section main result}).

\textbf{The Operator-splitting scheme.} We split the
dynamics described in \eqref{eq main PDE} by two phases:
\begin{enumerate}
    \item \textit{Conservative (transport) phase}: for a given $\rho$, we solve the conservative part, which is simply a transport equation, using the method of characteristics
    $$
    \partial_t \rho + \divv \big( \rho b[\rho]\big)=0
    $$
    The existence of a solution to the above equation under a transport/push-forward map is guaranteed by DiPerna Lions theory \cite{diperna1989ordinary}.
    \item \textit{Dissipative (diffusion) phase}: we solve the dissipative (diffusion) part using a JKO-miminimizing movement scheme
    \begin{equation}
    \label{eq: dissipative part}
    \partial_t\rho= \divv \big(A\big(\nabla \rho+\rho\nabla f)\big).    
    \end{equation}
   We emphasize again that we allow the diffusion matrix $A$ to be degenerate. Because of the degeneracy of $A$, the JKO-scheme \eqref{eq discrete general grad flow} using the $A$-weighted Wasserstein matrix developed in \cite{lisini2009nonlinear} is not applicable. To overcome this difficulty, we use a simple idea, that is to use a small perturbation of $A$ to get a symmetric positive definite matrix. The key novelty here is that we perturb $A$ by $A_h:=A+hI$ where $h$ is the time-step in the discretisation scheme. Note that this perturbation is equivalent to adding a small amount of noise into each component of the stochastic dynamics \eqref{eq SDE of McKean type}. Thereby, we solve the dissipative (diffusion) equation iteratively using the minimizing movement scheme: $\rho^{n+1}_h$ is determined as the unique minimizer of the minimization problem
   $$
   \min_\rho \Big\{\frac{1}{2h} W_{c_h}(\rho,\rho^{n}_h)+ \int f\rho+\rho\log\rho\Big\},
   $$
   where (see Section \ref{section main result}; $\Pi(\mu,\nu)$ denotes the set of transport plans between the measures $\mu,\nu$)
   $$
   W_{c_h}(\mu,\nu)=\inf_{\gamma\in \Pi(\mu,\nu)}\int_{\mathbb{R}^{2d}}\big\langle\, (A+hI)^{-1}(x-y), (x-y)\,\big\rangle \,d\gamma(x ,y).
   $$
\end{enumerate}
Our main result, Theorem \ref{Theorem Main}, establishes the convergence of the above splitting-scheme to a weak solution of \eqref{eq main PDE} as the time step $h$ tends to zero. Our operator-splitting scheme is simple and natural capturing the conservative-dissipative splitting of the dynamics, in particular the expectation that the dissipative part is a $A$-weighted Wasserstein gradient flow. Furthermore, motivated by the efficiency of entropic regularisation methods in computational performances, %(see Section \ref{sec Wasserstein gradient flow}), 
in Theorem \ref{theorem reg-main} we also provide an entropic regularisation of the above scheme. We expect that the entropic regularisation scheme will be useful when one performs numerical simulations although we do not pursue it here. Our result offers a unified approach to establish existence results for a wide class of degenerate, non-local, conservative-dissipative systems. In fact, the class of \eqref{eq main PDE} is rich and includes many cases of interest: the linear and kinetic Fokker-Planck \cite{risken1989fokker}, the (regularised) Vlasov Poisson Fokker-Planck \cite{hwang2011vlasov}, the linear Wigner Fokker-Planck \cite{arnold2012wigner}, and higher-order degenerate diffusions approximating the generalised Langevin and generalised Vlasov equations \cites{ottobre2011asymptotic,Duong2015NonlinearAnalysis}. We will discuss in details these concrete applications in Section \ref{section examples}.
\medskip

\textbf{Comparison to existing literature.} There is a vast literature on operator-splitting methods for solving PDEs, see e.g. \cites{GlowinskiMR3587821SplitMethodsBook}. %holden2010splitting  
We now compare our work with the most relevant literature where the dissipative dynamics involves a Wassertein-type gradient flow. The closest paper to ours are \cites{carlier2017convergence,bernton2018langevin} where the authors consider equations of the form \eqref{eq main PDE}  and introduce similar operator-splitting schemes. However, these papers are limited to non-degenerate diffusion matrices $A$ ($A=I$ in these papers). \textcolor{black}{In fact, \cite{bernton2018langevin} does not deal with mixed dynamics, the splitting is carried out at the level of the gradient flow}.  \textcolor{black}{In \cite{yao2013blow}, the authors implement a numerical method that splits an aggregation-diffusion equation, where they exploit its transport structure using a Lagrangian method for the aggregation part, and employ an implicit finite-difference scheme for the diffusion part. Our splitting method is of a different nature, in that we would treat \cite{yao2013blow}[Equation (1.1)] as a dissipative equation with no conservative dynamics.} Other works that also develop operator-splitting schemes for degenerate PDEs are \cites{carlen2004solution,marcos2020solutions}, however these works only deal with a linear, local conservative dynamics and use a rather involved splitting mechanism. Several papers including \cites{Huang00,duong2014conservative,DuongTran18, adams2021entropic} also develop JKO-type minimizing movement schemes for degenerate diffusion equations; however these papers use  one-step schemes where the cost functions are often non-homogeneous, time-step dependent and do not induce a metric. We also mention recent works in which operator-splitting methods have been investigated for partial differential equations containing a Wasserstein gradient flow part and a non-Wasserstein part. The papers \cites{bowles2015weak,DuongLu2018} construct operator-splitting schemes for fractional Fokker-Planck equations, in which the transport phase is solved by a JKO-type minimizing movement scheme while the fractional diffusion is solved exactly by convolution with the fractional heat kernel. More recently, \cite{Liu2021} builds operator-splitting scheme for reaction-diffusion systems with detailed balanced based on an energetic variational formulation of the systems.
\medskip

\textbf{Future work.}  From a modelling perspective the non-local term $b$ captures the interactions between a large ensemble of particles. In this case, it takes the form of a convolution between the density distribution and a certain kernel, and our assumptions require the kernel to be uniformly bounded and Lipschitz. However, many fundamental models of interacting particle systems compose of singular interaction kernels \cites{Jabin2018,serfaty2020mean}.
%In this paper, we demonstrate via the regularised Vlasov-Poisson-Fokker-Planck equation that our method is applicable when one regularises the Coulomb interaction (see Section \ref{sec reg VPFPE}) \gbox{not anymore}. 
This leads to the natural and challenging question: can our method be generalized to deal with singular interaction kernels? %\textcolor{black}{A less ambitious task would be to relax the uniformly bounded condition on the interaction kernel, see Lemma \ref{lemma lipschitz kernal}.} %\textcolor{red}{If} this could be done, say for the Coulomb interaction, then our method could be used to construct solutions to the Vlasov-Poisson-Fokker-Planck equation. 
  In this paper, we demonstrate via the regularised Vlasov-Poisson-Fokker-Planck equation that our method is applicable when one regularises the Coulomb interaction (see Section \ref{sec reg VPFPE})
%A simpler approach would be to regularise the kernel as in \cite{huang2000variational}, however it turns out our results would still not be applicable because the kernel would not be Lipschitz.
Another interesting question is whether we can use the variational structure developed in this paper to study exponential convergence to the equilibrium of degenerate PDEs of the form \eqref{eq main PDE}. This is related to the hypocoercivity theory introduced by Villani \cite{dric2009hypocoercivity}, and a method using variational structures would provide further physical insight to the theory. These themes are left for future work.

\subsection{Organisation of the paper} In Section \ref{section main result} we set up the notations and present the operator-splitting scheme, assumptions, and the main result of this paper. The proof of the main result is given in Section \ref{section proof of the main result}.  In Section \ref{section entropy regularisation} we show how the scheme can be regularised. Section \ref{section examples} provides several explicit examples to which our work can be applied to. Finally, the Appendix contains detailed computations and proofs of technical results.

\section{The operator-splitting scheme, assumptions and our main result}\label{section main result}
In this section, we first introduce notations that will be used throughout the paper, then we present the operator-splitting scheme and assumptions and finally, we state the main result of this paper, Theorem \ref{Theorem Main}.

\subsection{Notation}
\label{sec: notation}
% \textbf{constants}

Throughout $d\in \bN$ will be the dimension of the space. A fixed $T>0$ denotes the length of the time interval we consider. Throughout, $C$ denotes a constant whose value may change without indication and depends on the problem's involved constants, but, critically, it is independent of key parameters of this work, namely the time step $h>0$ and number of iterates  $N\in \bN$ of the scheme introduced in Section \ref{section introduction}. The Euclidean inner product between two vectors $x,y \in \bR^d$ will be written as $x \cdot y$  or sometimes $\langle x,y \rangle$. We write $\|\cdot\|$ as the Euclidean norm on $\bR^d$, and $|\cdot|$ when $d=1$. The symbol  $\|\cdot\|$ is also used as the $2$-norm on $\bR^{d\times d}$. For a matrix $A$ let $A^T$ be its transpose. 

% \textbf{Function spaces}

Let $\Omega \subseteq \bR^d $, we write $|\Omega|$ as its $d-$dimensional Lebesgue measure.
The space of Lebesgue $m-$integrable functions on $\Omega$ is denoted by $L^m(\Omega)$. The Sobolev space of functions in $L^1(\Omega)$ with first weak derivatives also in $L^1(\Omega)$ is denoted $W^{1,1}(\Omega)$. We say that $f\in L^{1}_{\text{loc}}(\bR^d)$ if $f\in L^1(\Omega)$ for any compact $\Omega\subset \bR^d$. We define the space $f\in W^{1,1}_{\text{loc}}(\bR^d)$ similarly.  The supremum norm $\|\cdot\|_{\infty,\Omega}$ of a vector field $\phi : \Omega \to \bR^d$, or a function $\phi : \Omega \to \bR$, is used to denote $\sup_{x\in \Omega}\|\phi(x)\|$, $ \sup_{x\in \Omega}|\phi(x)|$ respectively, when $\Omega=\bR^d$ we just write $\|\cdot\|_{\infty}$. We use the Landau ``big-O'' notation $\phi(h)=O(\varphi(h))$, for functions $ \phi,\varphi : \bR_+ \to \bR $ to mean that there exists $C,h_0>0$ such that $|\phi(h)|\leq C \varphi(h)$ for all $h<h_0$ and we say a matrix $B\in \bR^{d\times d}$ is $O(h)$ if $\max_{i,j}|B_{i,j}|\leq C h$. % Further we use the Landau ``little-o'' notation $\phi(h)=o(\varphi(h))$ to mean $\lim_{h\to 0} \frac{\phi(h)}{\varphi(h)}=0$.

Let $A,B\subseteq \bR^d$, define $C^k(A;B)$ as the $k$-times continuously differentiable functions from $A$ to $B$ with continuous $k^{th}$ derivative. Define $C^\infty_c(A;B)$ as the set of infinitely differentiable functions from $A$ to $B$ with compact support. \textcolor{black}{We specifically write  $C^\infty_c(\bR^d)$ to denote infinitely differentiable functions from $\bR^d$ to  $\bR$ with compact support. Let $C_b(\bR^d)$ be the set of continuous bounded functions from $\bR^d$ to $\bR$.}
% \textbf{Differentials}
Let $\nabla \phi$, $\Delta \phi$, and $\nabla^2 \phi$ be the gradient, Laplacian, and Hessian respectively, of a sufficiently smooth function $\phi : \bR^d \to \bR $. For a sufficiently smooth vector field $\eta : \bR^d \to \bR^d$ let $\divv (\eta)$, and $D\eta$ be its divergence and Jacobian respectively. We call `$\text{id}$' the identity map on any space. 

% \textbf{probability measures, and spaces}

Denote the space of Borel probability measures on $\bR^d$ as $\cP(\bR^d)$. The second moment $M$ of a measure $\rho \in \cP(\bR^d)$ is defined as
\begin{equation}\label{eq second moment definition}
    \cP(\bR^d) \ni \rho \mapsto 
M(\rho):=\int_{\bR^d}\|x\|^2\rho(dx).
\end{equation}
The set of probability measures with finite second moments is denoted $\cP_2(\bR^d)$, 
\begin{equation}
\label{eq: measure second moment space definition}
    \cP_2(\bR^d):=\{\rho \in \cP(\bR^d)~:~M(\rho)<\infty\}.
\end{equation}
Define $\cP_2^r(\bR^d)$ as those $\rho \in \cP_2(\bR^d)$ which are absolutely continuous. Throughout, when a measure is said to be `absolutely continuous' we implicitly mean with respect to the Lebesgue measure . We will use the same symbol $\rho$ to denote a measure $\rho \in \cP_2^r(\bR^d)$ as well as its associated density. Define $H$ to be the negative of Boltzmann entropy,
\begin{equation}\label{eq bolztman entropy definition }
    \cP(\bR^d) \ni \rho \mapsto 
H(\rho):= \begin{cases}
    \int_{\mathbb{R}^{d}} \rho \log \rho, &\text{if $\rho \in \cP^r_2(\bR^d)$}
    \\
    +\infty, & \textrm{otherwise,}
    \end{cases}
\end{equation}
 which throughout we will just refer to as the entropy. Also define the positive part of the entropy as 
 
 \begin{equation}
    \cP(\bR^d) \ni \rho \mapsto 
H_+(\rho):= \begin{cases}
    \int_{\mathbb{R}^{d}} \max\{\rho \log \rho,0\}, &\text{if $\rho \in \cP^r_2(\bR^d)$}
    \\
    +\infty, & \textrm{otherwise},
    \end{cases}
\end{equation}

and the negative part of the entropy as 

\begin{equation}
    \cP(\bR^d) \ni \rho \mapsto 
H_-(\rho):= \begin{cases}
    \int_{\mathbb{R}^{d}} | \min\{\rho \log \rho,0\}|, &\text{if $\rho \in \cP^r_2(\bR^d)$}
    \\
    +\infty, & \textrm{otherwise}.
    \end{cases}
\end{equation}

\color{black}
The set of transport plans between given measures $\mu,\nu \in \cP_2(\bR^d)$ is denoted by $\Pi(\mu,\nu)\subset \cP_2(\bR^{2d})$. That is, for $\mu,\nu \in \cP_2(\bR^d)$, $\gamma \in \Pi(\mu,\nu)$ if $\gamma(\mathcal{B}\times \bR^d)=\mu(\mathcal{B})$ and $\gamma( \bR^d \times \mathcal{B})=\nu(\mathcal{B})$ for all Borel sets $\mathcal{B}\subset \bR^d$. Let $\Pi^r(\mu,\nu)$ be those $\gamma \in \Pi(\mu,\nu)$ which are absolutely continuous. We denote a sequence of probability measures indexed by $k\in \bN$ as $\{\mu_k\}_{k\in \bN}$ which we relax to $\{\mu_k\}$. We use the symbol $\rightharpoonup$ to mean the weak  convergence of measures, that is $\rho_k \rightharpoonup \rho$ (weakly) if 

\begin{equation}\label{eq weak convergence definition}
\lim_{k \to \infty} \int f d\rho_k = \int f d\rho, ~~\forall f\in C_b(\bR^d).     
\end{equation}
We also recall that if it is known that $\rho \in \cP(\bR^d)$ then the weak convergence \eqref{eq weak convergence definition} is equivalent to narrow convergence, that is convergence against $C^\infty_c(\bR^d)$ functions. For any two subsets $P,Q\subset \cP_2(\bR^d)$ we denote $\Pi(P,Q)$ as the set of transport plans whose marginals lie in $P$ and $Q$ respectively. For a vector field $\eta:\bR^d\to \bR^d$ and measure $\mu\in\cP(\bR^d)$ we write $(\eta)_{\#}\mu$ as the push-forward of $\mu$ by $\eta$. For any probability measure $\gamma$ and function $c$ on $\bR^{2d}$ we write
$$
(c,\gamma):=\int_{\bR^{2d}} c(x,y) d\gamma(x,y).
$$
We use the symbol $*$ to denote the convolution, that is for a vector field $K:\bR^{d_1} \to \bR^{d_1}$ and a measure $\rho\in \cP(\bR^{d_2})$, $K*\rho : \bR^{d_1} \to \bR^{d_1}$ is defined as
\begin{equation}
    K*\rho(x):=\int_{\bR^{d_2}} K(x-x')\rho(x',z) dx'dz,
\end{equation}
where $x,x'\in \bR^{d_1}$ and $z\in \bR^{d_2-d_1}$. Lastly, the $2$-Wasserstein distance on $\cP_2(\bR^d)$ is denoted by $W_2$.

\subsection{The operator-splitting scheme}
\label{sec the scheme}
We now present our operator-splitting scheme for solving \eqref{eq main PDE}. Denote the free energy $\cF :\ \cP_2^r(\bR^d)\to \bR$ as the sum of the potential energy and the entropy
\begin{equation*}
    \cF(\rho):=F(\rho)+H(\rho),
\end{equation*}
where
\begin{equation*}
    F(\rho):=\int_{\bR^{d}} \rho f  dx,
    \quad\textrm{and}\quad  H(\rho):=\int_{\bR^{d}}\rho\log(\rho) dx.
\end{equation*}
\paragraph{Operator-splitting scheme:} Let $T>0$ denotes the terminal time and $\rho_0\in \cP_2^r(\bR^d)$ be given, with $\cF(\rho_0)<\infty$. Let $h>0$, $N\in\bN$ be such that $hN=T$, and let $n\in\{0,\ldots,N-1\}$. Set $\rho^0_h=\tilde{\rho}^0_h=\rho_0$. 
Given $\rho^n_h$, our operator-splitting to determine $\rho^{n+1}_h$ consists of two phases
\begin{enumerate}
    \item \textit{Conservative (transport) phase}:  first we introduce the push forward by the Hamiltonian flow as
\begin{equation}\label{eq scheme push forward}
    \tilde{\rho}^{n+1}_h=X^n_h(h,\cdot)_{\#}\rho^n_h,
\end{equation}
where $X_h^n : \bR^+\times \bR^d\to \bR^d$ solves the ODE 
\begin{equation}\label{eq definition of the flow}
\begin{cases}
&\partial_t X^n_h = b[\rho_h^n]\circ X^n_h,
\\
& X^n_h(0,\cdot)=\text{id}.
\end{cases}
\end{equation}
\item \textit{Dissipative (diffusion) phase}: next, define $\rho^{n+1}_h$ as the minimizer of the following JKO-type optimal transport minimization problem
\begin{equation}\label{eq the JKO step}
\rho^{n+1}_h
=\underset{\rho\in \cP_2^r(\bR^d)}{\text{argmin}}\Big\{ \frac{1}{2h}W_{c_h}(\tilde{\rho}_h^{n+1},\rho)+\cF(\rho) \Big\},
\end{equation}
where $W_{c_h}$ is a Kantorovich optimal transport cost functional, defined for $h>0$ as
\begin{equation}\label{eq : minimisation problem}
    W_{c_h}(\mu,\nu):= \inf_{\gamma \in \Pi(\mu,\nu)} \int c_h(x,y) d\gamma(x,y),
\end{equation}
with the cost function $c_h :\bR^{2d}\to \bR$ given by
\begin{equation}\label{eq weighted wasserstein}
    c_h(x,y):= \big\langle A_h^{-1}(x-y), (x-y)\big\rangle,
\end{equation}
for the matrix $A_h \in \bR^{d\times d}$ defined as
\begin{equation}\label{eq the matrix A}
    A_h:=A+hI.
\end{equation}
\end{enumerate}
Note that since $A$ is symmetric positive semi-definite,  the addition of $hI$ to $A$ guarantees that $A_h$ is symmetric positive definite (see Lemma \ref{lemma the cost function i.e matrix A}). Hence, $c_h$ is well defined for all $h>0$ and $\sqrt{c_h}$ defines a metric on $\bR^d$, which in-turn means $W^{1/2}_{c_{h}}$ defines a metric on $\cP_2(\bR^d)$. This is precisely a $A_h$-weighted Wasserstein distance \cite{lisini2009nonlinear}.  The above perturbation can be also effectively achieved by adding a small noise $\sqrt{2h}dW(t)$ to the SDE \eqref{eq SDE of McKean type}. We mention that if the matrix $A$ is invertible then there is no need to perform the perturbation. %\eqref{eq the matrix A}. 
Instead we can adopt the scheme with $c_h(x,y)=c(x,y):=\big\langle A^{-1}(x-y), (x-y)\big\rangle $ and all results would remain true. This is the case for the Linear Wigner Fokker-Planck, see Section \ref{sec example linear wigner}. 

For each $n\in \{0,\ldots, N\}$ we denote  $\tilde{\gamma}^{n,c}_h,\tilde{\gamma}^n_h \in \Pi(\tilde{\rho}^n_h,\rho^n_h)$, as the following optimal couplings (respectively)
\begin{equation}\label{eq optimal coupling 1}
    W_{c_{h}}(\tilde{\rho}^{n}_h,\rho^n_h)
    =\int_{\bR^{2d}} c_h(x,y) d\tilde{\gamma}^{n,c}_h(x,y),
    \qquad
    W^2_{2}(\tilde{\rho}^{n}_h,\rho^n_h)
    =\int_{\bR^{2d}} \|x-y\|^2 d\tilde{\gamma}^{n}_h(x,y),
\end{equation}
and for $n\in \{0,\ldots, N-1\}$ we define $\gamma^{n}_h \in \Pi(\rho^n_h,\tilde{\rho}^{n+1}_h)$ as the optimal coupling 
\begin{equation}\label{eq optimal coupling 2}
    W^2_{2}(\rho^{n}_h,\tilde{\rho}^{n+1}_h)=\int_{\bR^{2d}} \|x-y\|^2 d\gamma^{n}_h(x,y).
\end{equation}
The optimal couplings in \eqref{eq optimal coupling 1} and \eqref{eq optimal coupling 2} are all well defined, see Lemma \ref{lemma 25}. 
Throughout this work we will adopt the notation that $t_n=nh$ for $n\in\{0,\ldots, N\}$. Consider the following piece-wise constant in time  interpolations
%\begin{equation}
   % \rho_h(t,\cdot)=
   % \begin{cases}
  %  \tilde{\rho}_h^{n+1}~&\text{for}~t\in %[2nh,(2n+1)h)
  %  \\
 %   \rho_h^{n+1}~&\text{for}~t\in [(2n+1)h,2(n+1)h)
 %   \end{cases}
%\end{equation} 
of $\{\rho^n_h\}_{n=0}^{N}$
\begin{equation}\label{eq interpolation 1}
    \rho_h(t,\cdot):=\rho_h^{n+1}~\text{for}~t\in [t_n,t_{n+1}),
\end{equation}
and of $\{\tilde{\rho}^n_h\}_{n=0}^{N}$ 
\begin{equation}\label{eq interpolation 2}
\tilde{\rho}_h(t,\cdot):=   \tilde{\rho}_h^{n+1}~\text{for}~t\in [t_n,t_{n+1}),
\end{equation}
and \textcolor{black}{consider the interpolation of $\{\tilde{\rho}^n\}_{n=0}^N$, which continuously follows the conservative dynamics}
\begin{equation}\label{eq interpolation 3}
     \rho^{\dag}_h(t,\cdot):= \big(X^n_h(t-t_n,\cdot)\big)_{\#}\rho^n_h~\text{for}~t\in [t_n,t_{n+1}),
\end{equation}
so that for $t\in[t_n,t_{n+1})$, \textcolor{black}{ $\rho^{\dag}_h(t)=\mu(t-t_n)$} where $\mu$ is the solution of  the continuity equation (see Lemma \ref{lemma 12})
\begin{equation}\label{eq 5}
    \begin{cases}
    \partial_t \mu(t,\cdot) + \divv \big( \mu(t,\cdot) b[\rho_h^n]\big)=0
    \\
    \mu(t,\cdot)|_{t=0}=\rho^n_h.
    \end{cases}
\end{equation}
We now introduce assumptions on the potential $f$, the non-local vector field $b$, and the diffusion matrix $A$. Under these assumptions we will prove the well-posedness of the splitting scheme and the convergence of the interpolations \eqref{eq interpolation 1}-\eqref{eq interpolation 3} to a weak solution of \eqref{eq main PDE}.

\begin{assumption}
\label{assumption : main assumptionso on drift and diffusion matrix}

The potential energy $f\in C^1(\bR^d)$ is assumed to be non-negative $f(x) \geq 0$, and Lipschitz, that is there exists a constant $C>0$ such that for any $x,y\in\bR^d$ 
\begin{equation}
\label{assumption f non-negative and Lipschitz}
    |f(x)-f(y)|\leq C \|x-y\|.
\end{equation}
For the non-local drift $b:\cP(\bR^d)\times \bR^d \to \bR^d$, we assume that there exists $C>0$ such that for any $\mu \in \cP_2(\bR^d)$ 
\begin{equation}\label{eq assump 4}
    \| b[\mu](x) \|\leq C\big(1+\|x\|\big), ~\forall x \in \bR^d,
    \quad 
    b[\mu]\in W^{1,1}_{\text{loc}}(\bR^d),\qquad \divv(b[\mu])=0.
\end{equation}
Moreover, we assume there exists $C>0$ for all $\mu,\nu \in \cP_2(\bR^d)$
\begin{equation}\label{eq assump 3}
    \int_{\bR^d} \|b[\nu](x)-b[\mu](x)\|^2 d \nu(x) \leq C W_2^2(\nu,\mu).
\end{equation}
Lastly assume $A$ is a semi-positive definite (symmetric), constant matrix $A\in \bR^{d\times d}$.
\end{assumption}

\begin{rem}[Commenting on the assumptions]
The Lipschitz assumption on $f$ is standard when working on the space of probability measures with finite second moments, particularly ensuring that the free energy functional is well-defined. In terms of the assumptions on the non-local vector field $b$,  \eqref{eq assump 4} implies well-posedness of the transport problem via DiPerna-Lions theory \cite{diperna1989ordinary}. Moreover, imposing the regularity in the measure component \eqref{eq assump 3} allows us to obtain upper-bounds for some error terms when proving the convergence of the scheme to a weak solution of \eqref{eq main PDE}. Note that when $b$ takes the form of a convolution with an interaction kernel, \eqref{eq assump 3} is satisfied when the kernel is uniformly bounded, Lipschitz and differentiable, which are the cases for the examples in Section \ref{section examples}. Note that the above assumptions have been also made in \cite{carlier2017splitting}.
\end{rem}
We now make the definition of a weak solution to \eqref{eq main PDE} precise.
\begin{defn}[Weak solution]
\label{def: weak formulation general PDE}
The curve $\rho:[0,T]\to \cP_2^r(\bR^d)$, $t\mapsto \rho(t,\cdot)$, is called a weak solution to the general evolution equation \eqref{eq main PDE} if for all $\varphi \in C^\infty_c([0,T]\times \bR^d)$ we have 
\begin{align}
\label{eq weak solution of main eq}
\int_{0}^{T}\int_{\bR^d} \rho(t,x)\Big(\partial_t \varphi(t,x) + \big( b[\rho(t,\cdot)](x) \textcolor{black}{-} A \nabla f(x) \big) \cdot \nabla \varphi (t,x) + \divv \big(A \nabla \varphi (t,x)\big) \Big)dxdt + \int_{\bR^d} \rho^0(x) \varphi(0,x) dx& = 0
\end{align}
\end{defn}
The main (abstract) result of this work is the following theorem which gives the existence of weak solutions of the evolution equation \eqref{eq main PDE}. We do not deal with uniqueness here, but in principle, it can be obtained via displacement convexity arguments and an exponential in time contraction on the $W_2$ distance between two solutions started from different initial data, cf. \cite{laborde2016some}. 
\begin{theorem}%[Main theorem]
\label{Theorem Main}
Let $\rho_0\in \cP^r_2(\bR^d)$ satisfy $\cF(\rho_0)<\infty$. Let $h>0$, $N\in\bN$ with $hN=T$, and let $\{\rho^{n}_{h}\}_{n=0}^{N},\{\tilde{\rho}^{n}_{h}\}_{n=0}^{N}$ be the solution of the scheme \eqref{eq scheme push forward}-\eqref{eq the JKO step}. Define the piecewise constant interpolations $\rho_{h},\tilde{\rho}_h$ by \eqref{eq interpolation 1}-\eqref{eq interpolation 2} and the interpolation $\rho^\dag_h$ by \eqref{eq interpolation 3}. 
Suppose that Assumption \ref{assumption : main assumptionso on drift and diffusion matrix} holds. Then  
\begin{enumerate}[label=(\roman*)]
   \color{blue}
    \item  for each $t\in[0,T]$ as $h \rightarrow 0$ ($N \to \infty$ abiding by $hN=T$) we have
    
\begin{equation}
    \label{eq: convergence}
    \rho_{h}(t,\cdot),\tilde{\rho}_h(t,\cdot),\rho^\dag_h(t,\cdot)
    % \rho_k 
  \underset{h\to 0}{\longrightarrow} \rho(t) \quad\text{in}\quad L^1\big(\bR^d \big).
\end{equation}
\color{black}
\item Moreover,  there exists a map  $[0,T]\ni t \mapsto \rho(t,\cdot)$ in $\cP_2^r(\bR^d)$ such that 
\begin{equation}
    \label{eq: wasserstein sup convergence}
   \underset{h\to 0}{\lim} \sup_{t\in[0,T]} \max \Big\{ W_2( \rho_{h}(t,\cdot),\rho(t,\cdot)),  W_2(\tilde{\rho}_h(t,\cdot),\rho(t,\cdot)), W_2(\rho^\dag_h(t,\cdot),\rho(t,\cdot)) \Big\}=0.
\end{equation}
\end{enumerate}
The $\rho$ maps appearing in the above limits are weak solutions of the evolution equation \eqref{eq main PDE} in the sense of Definition \ref{def: weak formulation general PDE}.
\end{theorem}
\textcolor{blue}{
Note that the convergence \eqref{eq: convergence} is stronger than weak $L^1((0,T)\times \bR^d)$ convergence.} The proof of the above theorem is given in Section \ref{section proof of the main result}.

\begin{rem}\label{remark nonlinear}
If one were to instead consider the  evolution equation, for a non-linear function $P$,
\begin{equation*}
    \partial_t \rho + \divv ( \rho b[\rho])  =   \divv \Big(A\big(\nabla P( \rho)+\rho\nabla f \big)\Big),
\end{equation*}
then following the strategy in \cite{carlier2017splitting}, to deal with the non-linear term, we expect one could construct a similar scheme to the one detailed above by adjusting the free energy functional $\cF$. We leave this for now to not over complicate the presentation.
\end{rem}

\section{Proof of the Main Result}\label{section proof of the main result}
The objective of this section is to prove the main result, Theorem \ref{Theorem Main}. Once a suitable optimal transport cost functional has been identified, the proof of the convergence of the discrete variational approximation scheme to a weak solution of the evolutionary equation is now a well-established procedure following the celebrated strategy of \cite{jordan1998variational}: firstly we prove the well-posedness of the scheme, then we derive discrete Euler-Lagrange equations for the minimisers of \eqref{eq the JKO step} and  necessary a priori estimates, and finally we prove the convergence of the scheme to a weak solution of \eqref{eq main PDE}. An additional step in our proof for the constructed operator-splitting scheme is to combine the two (conservative and dissipative) phases together. Since the outcome of the conservative phase $\tilde{\rho}^{n+1}_h$ becomes an input of the dissipative phase, we need to show that the second moments, the free-energy functionals and the distances involved, with respect to this density are controllable. This is where we make use of the divergence-free property and assumptions of the non-local vector field $b$. To this end, our proof follows the methods in \cites{duong2014conservative, carlier2017splitting} and we will omit similar computations.

Recall from Section \ref{sec the scheme} the definitions of the sequences $\rho^n_h,\tilde{\rho}_h^n$, interpolations $\rho_h,\tilde{\rho}_h,\rho^\dag_h$, and optimal couplings $\tilde{\gamma}^{n,c}_h$, $\tilde{\gamma}^n_h$, $\gamma^n_h$. Also recall that the constant $C\geq 0$ that appears will be independent of $h$ and $n\in \{0,\ldots N\}$, but may depend on the final time $T$. The following results hold under the assumptions of Theorem \ref{Theorem Main}, and for all $0<h<1$, note that we are ultimately interested in the case where $h \to 0$.

\subsection{Preliminary results and well-posedness}
\label{section well-posed premliminary results}

The main result here is that the scheme proposed in Section \ref{sec the scheme} is well-posed, this is shown using the direct method of calculus of variations with respect to the weak topology. 
% the direct method of calc of variation : 

% \gbox{reword} We also affirm some of the results claimed in Section \ref{section introduction}, namely that the matrix $A_h$ is positive definite, and the flow of $b$ will solve the continuity equation.
We also make some preliminary observations on the matrix $A_h$, and on solutions to the continuity equation which will be useful later on.

\subsubsection{The transport equation}

By our assumptions on $b$, we can use DiPerna-Lions theory \cite{diperna1989ordinary} to conclude that there exists a solution to the ODE \eqref{eq definition of the flow}, which when pushing forward the initial density solves the continuity equation \eqref{eq 5}. 
Moreover, we note that the conservative dynamics preserves the entropy.
\begin{lemma}\label{lemma 12}
 Let $\rho^n_h \in \cP_2^r(\bR^d)$. Then the following results hold for any $n\in\{0,\ldots, N-1\}$.
\begin{enumerate}[label=(\roman*)]
    \item \label{item: 1} There exists a unique  $X_h^n:\bR_+\times \bR^d \to \bR^d$, such that $X_h^n(0,\cdot)=\text{id}$, and for a.e.~$x\in \bR^d$ the map $t\mapsto X_h^n(t,x)$ solves \eqref{eq definition of the flow},
        \begin{equation*}
        X_h^n(t,x)=x+\int_0^t b[\rho^n_h]\circ X^n_h(s,x)ds.
    \end{equation*}
    Moreover, $\bR^d \ni x \mapsto X(\cdot,x) \in L^1_{\text{loc}}(\bR^d;C(\bR))$, and for a.e.~$x\in \bR^d$ the map $\bR^+ \ni t \mapsto X_h^n(t,x) \in C^1(\bR)$.
    
    \item \label{item: 2}  For $t\in [t_n,t_{n+1})$, $\rho^\dag_h(t,\cdot)$ solves the continuity equation \eqref{eq 5} over the interval $[0,h)$.
    
    \item \label{item: 3} We have the following entropy preservation identities 
\begin{equation}\label{eq 14}
    H\big( \rho^{\dag}_h(t,\cdot) \big) = H(\rho^n_h)~ \forall t \in [t_n,t_{n+1}), \qquad  H(\tilde{\rho}^{n+1}_h)=H(\rho^n_h).
\end{equation}

    % see https://mathoverflow.net/questions/327781/prove-that-the-flow-of-a-divergence-free-vector-field-is-measure-preserving. OR MY notes (continuity equation)...
    
\end{enumerate}

\begin{proof}
Since $b[\rho^n_h]$ satisfies Assumption \ref{assumption : main assumptionso on drift and diffusion matrix}, \ref{item: 1} and \ref{item: 2} follow by \cite{diperna1989ordinary}*{Theorem III.1}. 
In regard to \ref{item: 3}, note that for all $t\geq 0$ the map  $X_h^n(t,\cdot)$ preserves the Lebesgue measure since $b$ is a divergence free vector field. The result is thus immediate.
\end{proof}
\end{lemma}

The following lemma bounds the change of the distribution under the Hamiltonian dynamics over the interval $(0,h)$ by its 2nd moments. 
\begin{lemma}\label{lemma 7} The following result holds for any $n\in\{0,\ldots, N-1\}$. Let $\rho^n_h \in \cP_2^r(\bR^d)$. Let $\mu$ be the solution of \eqref{eq 5} over the interval $[0,h]$ and let $0 \leq s_1 \leq s_2 \leq h$. 
Then
\begin{equation}\label{eq 12}
    W_2^2(\mu(s_1,\cdot),\mu(s_2,\cdot)\big) \leq C h \int_{s_{1}}^{s_{2}}  \big( 1+ M(\mu(s,\cdot)) \big)ds.
\end{equation}

Moreover, for any $t\in[t_n,t_{n+1})$,  $M(\rho^\dag_h(t,\cdot)),M(\tilde{\rho}_h(t,\cdot))< C\big(M(\rho^n_h)+1\big)$. 

\begin{proof}
Let $\mu$ solve \eqref{eq 5}. For any $ 0 \leq s_1\leq s_2 \leq h$, from Benamou-Brenier formula \cite{ambrosio2008gradient}*{Chapter 8} and  \eqref{eq assump 4},  we have 
\begin{align*}
    W_2^2\big( \mu(s_1,\cdot),\mu(s_2,\cdot) \big)\leq& (s_2-s_1) \int_{s_{1}}^{s_{2}} \int_{\bR^d} \| b[\rho^n_h](x) \|^2 \mu(s,x) dxds
    \\
    \leq& (s_2-s_1) C \int_{s_{1}}^{s_{2}} \int_{\bR^d} \big(1+\|x\|^2\big) \mu(s,x) dxds
 %   \\
 %   =& (s_2-s_1) C \int_{s_{1}}^{s_{2}}  \big(1+M(\mu(s))\big) ds
    % \\
    \leq h C \int_{s_{1}}^{s_{2}}  \big(1+M(\mu(s,\cdot))\big) ds,
\end{align*}
which is \eqref{eq 12}. Now consider
% Now for $t\in [t_n,t_{n+1})$, $\rho^\dag_h(t,x)=\rho^n_h((X^n_h(t-t_n,\cdot))^{-1}(x))$, and we have the moment bound
\begin{align*}
    \partial_t M(\mu(t,\cdot)) % = &\partial_t \int_{\bR^d} \|x\|^2\big( X^n_h(t,\cdot)\big)_{\#}\rho^n_h(x)dx \\
     =
    %  &
     \partial_t \int_{\bR^d} \|  X^n_h(t,x)\|^2\rho^n_h(x)dx
    % \\
     =& 2 \int_{\bR^d} X^n_h(t,x) \cdot \partial_t X^n_h(t,x)   \rho^n_h(x) dx
    \\
    =& 2 \int_{\bR^d} X^n_h(t,x) \cdot b[\rho^n_h] \circ X^n_h(t,x)   \rho^n_h(x) dx 
  %  \\
 %   \leq& 2 \int_{\bR^d} \| X^n_h(t,x)\| \|b[\rho^n_h] \circ X^n_h(t,x) \|  \rho^n_h(x) dx
%    \\
 %   \leq&  C \int_{\bR^d} \| X^n_h(t,x)\| \big(1+ \| X^n_h(t,x)\| \big)  \rho^n_h(x) dx
    \\
    \leq& C\int_{\bR^d}  \big(1+ \| X^n_h(t,x)\|^2 \big)  \rho^n_h(x) dx
    \\
    =& C\int_{\bR^d}  \big(1+ \|x\|^2 \big)  \big(X^n_h(t,\cdot)\big)_{\#}\rho^n_h(x) dx 
    % \\
    =
    % & 
    C\big(1+M(\mu(t,\cdot))\big).
\end{align*}

Employing Gr\"onwall's inequality, we have \textcolor{black}{for any $t\in[0,h]$ (recalling that throughout this article $h\leq T$)}
\begin{align}\label{eq 94}
    M(\mu(t,\cdot))
    \leq C\big(  M(\mu(0,\cdot))+1 \big)
       = C\big(M(\rho^n_h)+1\big).
\end{align}
For $t\in [t_{n},t_{n+1})$, recall  $\rho_h^\dag(t,\cdot)$ is equal to the solution of \eqref{eq 5} over the interval $[0,h)$. Hence for any $t\in [t_{n},t_{n+1})$, by \eqref{eq 94}, 
\begin{align*}
    M(\rho^\dag_h(t,\cdot))\leq C(M(\rho^n_h)+1).
\end{align*}
Moreover, for all $t\in[t_{n},t_{n+1})$ we have  $\tilde{\rho}_h(t,\cdot)=\tilde{\rho}^{n+1}_h=\mu(h,\cdot)$, where again $\mu$ solves \eqref{eq 5}, and hence by \eqref{eq 94}
\begin{align*}
    M ( \tilde{\rho}_h(t,\cdot) )
    =M(\tilde{\rho}^{n+1}_h)
    = M(\mu(h,\cdot))
    \leq C(M(\rho^n_h)+1),
\end{align*}
for any $t\in [t_{n},t_{n+1})$. This completes the proof.
%\begin{align*}
 %   \partial_t M(\mu_t) 
%    \textcolor{red}{=}&
 %   \int\|x\|^2 \partial_t \mu_t dx
 %   \\
%    =&-\int \divv (\mu_t W) dx
%    \\
%    =& 
%\end{align*}
\end{proof}
\end{lemma}

\subsubsection{The optimal transportation problem}

In this section we discuss the well-posedness of the minimization problem \eqref{eq the JKO step}. 
It is natural to achieve well-posedness of the scheme through finiteness, lower semi-continuity, and convexity of the functionals which appear in it.  First observe that $A_h$ is indeed positive.

\begin{lemma}[The cost function]\label{lemma the cost function i.e matrix A}
The matrix $A_h$ defined in \eqref{eq the matrix A} is positive definite (i.e., invertible) which implies, 
\begin{equation}\label{eq 2}
    \|x-y\|^2\leq C  c_h(x,y),~~\forall x,y \in \bR^d.
\end{equation}

\begin{proof} This is well-known \cite{adams2021entropic}*{Appendix B.1} .

\end{proof}

\end{lemma}
The next result addresses the existence of a unique minimiser to \eqref{eq the JKO step}. This type of result is classical, for completeness the details of the proof can be found in Appendix \ref{sec:appendix:wellposedness}. 
\begin{prop}
\label{lemma : well-posedness of jko scheme} 
Let $\mu \in \cP^r_2(\bR^d)$ with $\mathcal{F}(\mu)<\infty$. Then, there exists a unique $\nu^* \in \cP^r_2(\bR^d)$ such that 
\begin{equation}\label{eq the minimiser}
   \nu^*  = \underset{\nu \in \cP^r_2(\bR^d)}{\text{argmin}} \Big\{ \frac{1}{2h}W_{c_{h}}(\mu,\nu)+ \mathcal{F}(\nu) \Big\}.
\end{equation}
\end{prop}

\subsection{Discrete Euler-Lagrange Equations}\label{section E.L eq}
The following results are by now classical \cite{jordan1998variational}*{Proposition 4.1}, so we state them without proof. 
\begin{lemma}\label{lemma E.L eq 1}
Let $\eta \in C^\infty_c(\bR^d,\bR^d)$, and let $\Phi$ be the solution of the following ODE:
\begin{equation}\label{eq : the flow}
    \partial_s \Phi_s = \eta(\Phi_s),~\Phi_0=\text{id}.
\end{equation} 
Then for any $\nu \in \cP_2^r (\bR^d)$ we have 
\begin{equation}
\label{eq:aux:gradcalF}
    \delta \cF(\nu,\eta)
    :=\frac{d}{d s} \cF\big((\Phi_{s})_{\#}\nu\big)\Big|_{s=0}
    =\int_{\bR^d} \nu(y) \eta(y) \cdot  \nabla f(y)  dy-\int_{\bR^d} \nu(y) \divv (\eta (y)) dy.
\end{equation}

\end{lemma}

\begin{lemma}\label{lemma E.L eq 2}
 Let $\mu\in\cP_2^r(\bR^d)$. %and $h,\varepsilon$ be small enough.
 Let $\nu$ be the optimal solution in \eqref{eq the minimiser}, and let $\gamma$ be the corresponding  optimal plan in $W_{c_h}(\mu,\nu)$. 
 Then, for any $\eta \in C_c^\infty(\bR^d,\bR^d)$ we have

\begin{equation}
  0=\frac{1}{2h}\int_{\bR^{2d}} \Big\langle \eta(y) , \nabla_y   c_h(x,y)  \Big\rangle  d\gamma(x,y) +\delta \cF(\nu,\eta).
\end{equation}
In particular, for any $\varphi \in C^\infty_c(\bR^d)$, by choosing $\eta(x)=A_h \nabla \varphi(x)$, and $\tilde{\gamma}^{n+1,c}_h$ defined in \eqref{eq optimal coupling 1}, we have

\begin{equation}
  0=\frac{1}{h}\int_{\bR^{2d}} \Big\langle y-x, \nabla\varphi(x)  \Big\rangle  d\tilde{\gamma}_h^{n+1,c}(x,y) 
  +\delta \cF(\rho^{n+1}_h,A_h \nabla \varphi).
  \label{eq euler lagrange}
\end{equation}

\end{lemma}

\subsection{A priori estimates}

In this section we establish a priori estimates which will allow us to prove the convergence of the scheme to a weak solution of \eqref{eq main PDE} in Section \ref{section convergence}. More precisely, we will show the uniformly boundedness of the second moments and of free energies of the minimization iterations \eqref{eq the JKO step}. These uniform bounds are preserved under the conservative dynamics, this is explained in the next lemma. 

\begin{lemma}\label{lemma 4}
Let $n\in \{0,1,\ldots,N-1\}$. If there exists a constant $C_1>0$, independent of $h$ and $n$, such that $M(\rho^n_h),\cF(\rho^n_h)<C_1$, then $\tilde{\rho}^{n+1}_h$ obtained from \eqref{eq scheme push forward} satisfies 
$$
M(\tilde{\rho}^{n+1}_h),\cF(\tilde{\rho}^{n+1}_h)<C.
$$
As usual, the constant $C$ appearing is also independent of $h$ and $n$, but will depend on $C_1$.
\begin{proof}
The bound for the moments clearly hold by Lemma \ref{lemma 7}. For the free energy functional, we have $ \cF(\tilde{\rho}^{n+1}_h)= F(\tilde{\rho}^{n+1}_h)+H(\rho^n_h)$ by the conservation of entropy in Lemma \ref{lemma 12}. Therefore, since $f$ is  Lipschitz 
\begin{align*}
    \cF(\tilde{\rho}^{n+1}_h)
    =\int f(x)\tilde{\rho}^{n+1}_h(x)dx +H(\rho^n_h)
    & \leq C\int \big(\|x\|+1\big)\tilde{\rho}^{n+1}_h(x)dx +H(\rho^n_h)
    \\
    & \leq C(M(\tilde{\rho}^{n+1}_h)+1) +H(\rho^n_h) 
    % \\
    % & 
    \leq C(M(\tilde{\rho}^{n+1}_h)+1) +\cF(\rho^n_h) \leq C.
\end{align*}
\end{proof}
\end{lemma}

The following lemma controls the sum of the optimal transport costs of the JKO steps, by using $\tilde{\rho}^{n+1}_h$ as a competitor to $\rho^{n+1}_h$ in \eqref{eq the JKO step}. This estimate is of a similar type to \cite{jordan1998variational}[Equation (46)], however, because of the splitting nature of our scheme, we don't use $\rho^n_h$ as a competitor in \eqref{eq the JKO step}- making the estimate more involved. 

\begin{lemma}\label{lemma a sum of transport costs}
For any $n \in \{1,\ldots, N-1\}$ it holds that 

\begin{align}\label{z12}
 \sum_{i=0}^{n-1} W_{c_{h}}(\tilde{\rho}^{i+1}_h,\rho^{i+1}_h) \leq & C h\Big(1  + \cF(\rho^0)+\big(M(\rho^n_h)+1\big)^\alpha \Big).      
 \end{align}
\begin{proof}
Let $n\in \{0,1,\ldots,N-1\}$. 
Since $\rho^{n+1}_h$ attains the infimum in \eqref{eq the JKO step} we can compare it against $\tilde{\rho}^{n+1}_h$. This gives
\begin{equation*}
    \frac{1}{2h}W_{c_{h}}(\tilde{\rho}^{n+1}_h,\rho^{n+1}_h)\leq \cF(\tilde{\rho}^{n+1}_h)-\cF(\rho^{n+1}_h).
\end{equation*}
Using Lemma \ref{lemma 12} for the entropy, the above is equivalent to
\begin{equation}\label{eq 1}
    \frac{1}{2h}W_{c_{h}}(\tilde{\rho}^{n+1}_h,\rho^{n+1}_h)\leq F(\tilde{\rho}^{n+1}_h)-F(\rho^{n+1}_h)+ H(\rho^n_h)-H(\rho^{n+1}_h).
\end{equation}
Recall now that $(c_h,\tilde{\gamma}^{n+1,c}_h)=W_{c_h}(\tilde{\rho}^{n+1}_h,\rho^{n+1}_h)$. Using that $f$ is Lipschitz and Young's inequality with $\sqrt{\sigma}$ for some $\sigma>0$, we can see 
\begin{align*}
   F(\tilde{\rho}^{n+1}_h)-F(\rho^{n+1}_h) =
%   & 
    \int_{\bR^{2d}} (f(x)-f(y))d \tilde{\gamma}^{n+1,c}_h(x,y)
    % \\
\leq & C \int\|x-y\| d \tilde{\gamma}^{n+1,c}_h(x,y)
\\
\leq & \frac{C}{2 \sigma } \int \|x-y\|^2 d \tilde{\gamma}^{n+1,c}_h(x,y)+\frac{C\sigma }{2}
\\
\leq & \frac{C}{2\sigma }  \int c_h(x,y) d\tilde{\gamma}^{n+1,c}_h(x,y)  +\frac{C\sigma }{2},
\end{align*}
where in the last step we used \eqref{eq 2}. %\textcolor{red}{we dont need the moments here}...
Substituting this into \eqref{eq 1} yields 
 \begin{align*}
      \frac{1}{2h}W_{c_{h}}(\tilde{\rho}^{n+1}_h,\rho^{n+1}_h)\leq & \frac{C}{2\sigma }  \int c_h(x,y) d\tilde{\gamma}^{n+1,c}_h(x,y) +\frac{C\sigma }{2}
    %   \\
    %   &
      + H(\rho^n_h)-H(\rho^{n+1}_h).
 \end{align*}
Choosing $\sigma=2Ch$ leads to  
 \begin{align*}
      \frac{1}{2h}W_{c_{h}}(\tilde{\rho}^{n+1}_h,\rho^{n+1}_h) %\leq & \frac{1}{4 h}   \int c_h(x,y) d\tilde{\gamma}^{n+1,c}_h(x,y) +C h
    %   \\
    %   &
%      + H(\rho^n_h)-H(\rho^{n+1}_h),
      \leq &\frac{1}{4 h}  W_{c_{h}}(\tilde{\rho}^{n+1}_h,\rho^{n+1}_h) +C h+H(\rho^n_h)-H(\rho^{n+1}_h).
 \end{align*}
After rearranging we finally conclude 
 \begin{align}
 \label{eq 56}
    W_{c_{h}}(\tilde{\rho}^{n+1}_h,\rho^{n+1}_h) \leq & C\Big(  h^2+ h\big(H(\rho^n_h)-H(\rho^{n+1}_h)\big) \Big).
 \end{align}
 %using the bounded moments of Lemma \ref{lemma 3} and 
The sum of \eqref{eq 56} over $i\in\{0,\ldots,n-1\}$ contains a telescopic component which allows for the simplified expression
 \begin{align*}
 \sum_{i=0}^{n-1} W_{c_{h}}(\tilde{\rho}^{i+1}_h,\rho^{i+1}_h) \leq & C h\Big(1  + H(\rho^0)-H(\rho^n_h) \Big),      
 \end{align*}
where we have used that $Nh=T$. To deal with the terms in the right hand side of the above expression, we employ  \eqref{eq appendix 1} to deal with $H(\rho^i_h)$, while for $H(\rho^0)$ we just use the positivity of $f$. This leads to

  \begin{align*}
 \sum_{i=0}^{n-1} W_{c_{h}}(\tilde{\rho}^{i+1}_h,\rho^{i+1}_h) \leq & C h\Big(1  + \cF(\rho^0)+\big(M(\rho^n_h)+1\big)^\alpha \Big).      
 \end{align*}
\end{proof}
\end{lemma}

The next Lemma provides uniform bounds in $n$ and $h$ for the second moments, free energy functionals, and positive part of the entropy functionals of the solutions from the scheme \eqref{eq scheme push forward} - \eqref{eq the JKO step}. The proof is inspired by the procedure found in  \cite{adams2021entropic,duong2014conservative,Huang00}, first obtaining bounds locally and then extending them over the full time interval.

\begin{lemma}[Boundedness of the energy functionals, second moments and the positive part of the  entropy functionals]\label{lemma 3}
For all $n\in \{0,1,\ldots,N\}$,  we have
\begin{equation*}
    M(\rho^n_h),\cF(\rho^n_h),H_+(\rho^n_h) \leq C
% \end{equation}
\qquad \textrm{and}\qquad  
% \begin{equation}
    M(\tilde{\rho}^n_h),\cF(\tilde{\rho}^n_h),H_+(\tilde{\rho}^n_h) \leq C.
\end{equation*}
\begin{proof} 
Throughout this proof the constant $\bar{C}$ will change from line to line, and importantly it is independent of $\rho^0$. \\

For any $n\in\{1,\ldots,N\}$ we have that 

\begin{align}
    M(\rho^n)\leq 2\Big(M(\rho^0)+W^2_2(\rho^0,\rho^n)\Big)
    \leq& 2\Big( M(\rho^0) + n\sum^{n-1}_{i=0} W^2_2(\rho^i,\rho^{i+1})  \Big)\notag
    \\
    \leq& 4\Big( M(\rho^0) + n\sum^{n-1}_{i=0}  W^2_2(\rho^i,\tilde{\rho}^{i+1})+W^2_2(\tilde{\rho}^{i+1},\rho^{i+1})  \Big)\notag
    \\
    \leq& 4\Big( M(\rho^0) + n\sum^{n-1}_{i=0}  W^2_2(\rho^i,\tilde{\rho}^{i+1})+\bar{C}W_{c_h}(\tilde{\rho}^{i+1},\rho^{i+1})  \Big).
    \label{z13}
\end{align}
From Lemma \ref{lemma 7} we have  $4T W_2^2(\rho^i,\tilde{\rho}^{i+1})\leq \bar{C}h^2(1+M(\rho^i))$ for a constant $\bar{C}$ (independent of the initial condition), substituting this, and the bound \eqref{z12} into \eqref{z13}, we have, whilst noting $hN=T$,
\begin{align}
    M(\rho^n)\leq& 4 \Big( M(\rho^0) +  \bar{C}\big(1+\cF(\rho^0)\big) \Big) + \bar{C}\Big((1+M(\rho^n))^\alpha+h \sum_{i=0}^{n-1} (1+M(\rho^i))\Big)
    \nonumber
   \\
   \leq& C + \bar{C}\Big((1+M(\rho^n))^\alpha+h \sum_{i=0}^{n-1} (1+M(\rho^i))\Big)
    \nonumber
    \\
    \leq&  C + \bar{C}\Big((1+M(\rho^n))^\alpha+h \sum_{i=0}^{n-1} M(\rho^i)\Big),
    \label{p0}
    \end{align}
 for a constant $C$ depending only on $\cF(\rho^0)$ and $M(\rho^0)$, and constant $\bar{C}$ independent of $\rho^0$. Since the $\bar{C}$ appearing in \eqref{p0} is fixed and independent of the initial condition, we can find $h_0>0,N_0\in \bN$ (independent of the initial condition) such that for all $h\leq h_0$ we have $h N_0 \bar{C}\leq \frac{1}{2}$. Set $M_{N_{0}}:=\underset{n=1,\ldots, N_{0}}{\text{max}}M(\rho^n)$. Then \eqref{p0} implies 
 
 \begin{align}
     M_{N_{0}}\leq& C+ \bar{C}\Big((1+ M_{N_{0}})^\alpha+hN_0 M_{N_{0}}\Big)
     \nonumber
     \\
     \leq&C+ \bar{C}(1+ M_{N_{0}})^\alpha+ \frac{1}{2}M_{N_{0}},
     \nonumber
 \end{align}
    which implies 
     \begin{align}
     M_{N_{0}}\leq 2\big(C+ \bar{C}(1+ M_{N_{0}})^\alpha\big),
 \end{align}
 from which we can conclude $M(\rho^n)\leq C$, for all $n=1,\ldots,N_0$, and all $h\leq h_0$.  For the free energy, note that by definition of $\rho^{i+1}$, we have that 
\begin{align*}
    \cF(\rho^{i+1})-F(\tilde{\rho}^{i+1})-H(\tilde{\rho}^{i+1})\leq 0,
\end{align*}
adding and subtracting $F(\rho^i)$, and recalling that $H(\rho^i)=H(\tilde{\rho}^{i+1})$, implies
\begin{align}\label{z17}
    \cF(\rho^{i+1})-\cF(\rho^{i})&\leq |F(\rho^{i})-F(\tilde{\rho}^{i+1})|.
\end{align}
Summing \eqref{z17} from $i=0,\ldots, n-1$, using that $f$ is Lipschitz, and applying Young's inequality for some $\sigma>0$, we have

\begin{align}
    \cF(\rho^{n})-\cF(\rho^{0})\leq \sum_{i=0}^{n-1} |F(\rho^{i})-F(\tilde{\rho}^{i+1})|
  \leq C \sum_{i=0}^{n-1} \int_{\bR^{2d}} \|x-y\|d\gamma^i(x,y)
  \leq
    C \sum_{i=0}^{n-1}\Big( \frac{1}{\sigma} W_2^2(\rho^i,\tilde{\rho}^{i+1})+\sigma \Big).
    \label{z18}
\end{align}

Now let $N_0,h_0$ be chosen as before, and let $n=1,\ldots,N_0$. We know, by \ref{lemma 1} and the bounded moments just proved, that $W_2^2(\rho^i,\tilde{\rho}^{i+1})\leq Ch^2(1+M(\rho^i))\leq C h^2$ for $i\leq n$. Therefore, choosing $\sigma=h$ in \eqref{z18} implies the uniform bounded energies $\cF(\rho^n)\leq C$. Note that $\cF(\rho^n)\leq C$ implies $H(\rho^n)\leq C$, moreover, \eqref{eq appendix 1} and the uniform bounds on $M(\rho^n_h)$ imply that $H_-(\rho^n) \leq C$, therefore we have that $H_+(\rho^n)\leq C$. So far we have established the uniform bounds 

\begin{equation}\label{eq uniform moments new proof}
    M(\rho^n),\cF(\rho^n),H_+(\rho^n)\leq C, ~~~~\forall n=1,\ldots,N_0,~h\leq h_0.
\end{equation}
Since the $N_0$ and $h_0$ we have chosen are independent of the initial data we can extend the bound \eqref{eq uniform moments new proof} to all $n\in \{1,\ldots, N\}$ similarly as has been done in \cite[Lemma 5.3]{Huang00}, see also \cite{duong2014conservative}. The uniform bounds  $M(\rho^n_h),\cF(\rho^n_h),H_+(\rho^n_h) \leq C$, Lemma \ref{lemma 4}, and another application of \eqref{eq appendix 1} establishes   $M(\tilde{\rho}^n_h),\cF(\tilde{\rho}^n_h),H_+(\tilde{\rho}^n_h) \leq C$, completing the proof.

\end{proof}
\end{lemma}

\color{black}

Lemma \ref{lemma 3} states the uniform bounds for the discrete elements of our schemes. The following Lemma induces those bounds for the interpolations \eqref{eq interpolation 1}, \eqref{eq interpolation 2} and \eqref{eq interpolation 3}.

\begin{lemma}[A priori estimates for the interpolations]\label{lemma 8}
For all $n\in \{0,1,\ldots,N\}$, the moments, free-energies and the positive part of the entropies are uniformly bounded (in $n,h,t$), namely,
$$
M(\rho_h(t,\cdot)),M(\tilde{\rho}_h(t,\cdot)),M(\rho^{\dag}_h(t,\cdot)),\cF(\rho_h(t,\cdot)),\cF(\tilde{\rho}_h(t,\cdot)),\cF(\rho^{\dag}_h(t,\cdot)),H_+(\rho_h(t,\cdot)),H_+(\tilde{\rho}_h(t,\cdot)),H_+(\rho^\dag_h(t,\cdot))\leq C.
$$

\begin{proof}
These results for the interpolations follow easily from Lemma \ref{lemma 3}. Indeed, it is immediate from their definitions how this is inferred for the interpolations $\rho_h(t,\cdot),\tilde{\rho}(t,\cdot)$. For $\rho^\dag_h(t,\cdot)$, just notice from Lemma \ref{lemma 7} that we have  $M(\rho^\dag_h(t,\cdot))\leq C$. This uniform moment bound gives us the other two bounds for $\rho^\dag_h(t,\cdot)$: for the free energy one follows the argument in Lemma \ref{lemma 4} (using the bounded energy of $\rho^n_h$), and for the positive entropy one uses again \eqref{eq appendix 1}.
\end{proof}
\end{lemma}

The uniform bounds established in Lemma \ref{lemma 3} allow us to control the transport cost (w.r.t to both cost functions $c_h$ and $\|\cdot\|^2$) of the JKO step.

\begin{lemma}[Estimates of the sum of optimal transport costs]\label{lemma 5}
We have 
\begin{equation}\label{eq 67--eq 10}
    \sum_{n=0}^{N-1}W_{c_{h}}(\tilde{\rho}^{n+1}_h,\rho^{n+1}_h)\leq C h,
% \end{equation}
\quad \textrm{and} \quad
% \begin{equation}\label{eq 10}
    \sum_{n=0}^{N-1} W_2^2(\tilde{\rho}^{n+1}_h,\rho^{n+1}_h) \leq C h.
\end{equation}
% for a constant $C>0$.
\begin{proof}
The estimate \eqref{z12}, together with the uniform bounds of Lemma \ref{lemma 3}, gives the first result of \eqref{eq 67--eq 10}. The second result is immediate from the first and \eqref{eq 2}, since 
\begin{align*}
     \sum_{n=0}^{N-1} W_2^2(\tilde{\rho}^{n+1}_h,\rho^{n+1}_h) \leq   C \sum_{n=0}^{N-1} W_{c_h}(\tilde{\rho}^{n+1}_h,\rho^{n+1}_h) \leq C h.
\end{align*}
\end{proof}
\end{lemma}

The uniform moment bounds (in conjunction with the preliminary observation of Lemma \ref{lemma 7}) allow us to control the Wasserstein cost of the conservative phase.

\begin{lemma}[Estimates of the sum of optimal transport costs for the conservative dynamics]\label{lemma 10}
We have 
\begin{equation}\label{eq 11}
    \sum_{n=0}^{N-1} W_2^2(\rho^{n}_h,\tilde{\rho}^{n+1}_h) \leq C h.
\end{equation}
\end{lemma}
\begin{proof}
We recall \eqref{eq 12}, implying that for any $n=0,\ldots, N-1$,       $W_2^2(\rho^n_h,\tilde{\rho}_h^{n+1})\leq C h^2(1+M(\rho_h^n))$, the uniform bounded moment estimates then give the result.
\end{proof}
\color{black}

\subsection{Convergence of the operator-splitting scheme}
\label{section convergence}
Having obtained a priori estimates, in this section we prove the main theorem, Theorem \ref{Theorem Main}, that is the convergence of the time-interpolations of the discrete solutions constructed from the operator-splitting scheme in Section \ref{sec the scheme} to a weak solution of the main evolutionary equation \eqref{eq main PDE}. The following Lemma shows that these interpolations converge to limits which are equal almost everywhere to some curve $[0,T] \ni t \mapsto \rho(t,\cdot)\in \cP_2^r(\bR^d)$, and moreover, the sequences $\rho_h(t,\cdot),\tilde{\rho}_h(t,\cdot),\rho^\dag_h(t,\cdot)$ converge in $W_2$ to $\rho$, uniformly in time.

\begin{lemma}\label{lemma 9}[Convergence of the time-interpolations] There exists a curve $[0,T]\ni t \mapsto \rho(t,\cdot) \in \cP_2^r(\bR^d)$, such that % The sequences $\rho_h, \tilde{\rho}_h, \rho^\dag_h$ converge uniformly on $[0,T]$ in $W_2$ to $\rho\in \cP^r_2(\bR^d)$:
\begin{equation}
    \underset{h\to 0}{\lim} \sup_{t\in[0,T]} \max \Big\{ W_2( \rho_{h}(t,\cdot),\rho(t,\cdot)),  W_2(\tilde{\rho}_h(t,\cdot),\rho(t,\cdot)), W_2(\rho^\dag_h(t,\cdot),\rho(t,\cdot)) \Big\}=0.
\end{equation}
\begin{proof}
The proof follows an adapted version of \cite{ambrosio2008gradient}*{Theorem 11.1.6},  since we will require uniform in time convergence we utilise a refined version of the Arzela-Ascoli theorem \cite{ambrosio2008gradient}*{Proposition 3.3.1} in a similar manner to \cite{carlier2017splitting}[Section 5.1]. We provide the argument for $\rho_h$ only, the approach for $\tilde{\rho}_h,\rho^\dag_h$ is similar. To obtain the uniform convergence we set up an Arzela-Ascoli argument. Since the paths $\rho_h$ are not continuous we introduce the continuous concatenation of $\{\rho^n_h\}_n$ by geodesics. 
Let $n\in\{1,\ldots, N-1\}$. Fix any $s,t\in [0,T]$, define the path $\nu_h:[0,T]\to \cP_2(\bR^d)$ by concatenating $\rho^n_h$ and $\rho^{n+1}_h$ on $[t_{n-1},t_{n}]$ by a constant speed geodesic. Then for  $t\in [t_{n-1},t_n)$
\begin{align*}
W_2(\rho_h(t),\nu_h(t))
= W_2(\rho^{n}_h,\nu_h(t))
= & W_2(\nu_h(t_{n-1}),\nu_h(t)) 
\\
\leq & W_2(\rho^{n}_h,\rho_h^{n+1})(t-t_{n-1})\leq  W_2(\rho^{n}_h,\rho_h^{n+1}) h \leq C h, 
\end{align*}
by the bounded moments of Lemma \ref{lemma 8}. 
Let $t<s$, $t\in[ih,(i+1)h)$, $s\in[jh,(j+1)h)$, for some $i,j\in \{0,\ldots,N-1\}$. We then have 
\begin{align*}
    W_2(\nu_h(t),\nu_h(s))\leq& W_2(\nu_h(t),\rho^{i+1}_h)+W_2(\rho^{i+1}_h,\rho^{j+1}_h)+W_2(\rho^{j+1}_h,\nu_h(s))
    \\
    \leq&  h W_2(\rho^{i+2}_h,\rho^{i+1}_h)+W_2(\rho^{i+1}_h,\rho^{j+1}_h)+ h W_2(\rho^{j+1}_h,\rho^{j+2}_h)
    \\
    \leq& Ch + W_2(\rho^{i+1}_h,\rho^{j+1}_h).
\end{align*}
Using the triangle property and then the Cauchy Schwartz inequality we have 
\begin{align*}
W_2(\rho^{i+1}_h,\rho^{j+1}_h)\leq& \sum_{n=i+1}^{j}W_2(\rho_h^{n},\rho_h^{n+1})
% \\
\leq
% & 
\sqrt{(j-i)h}\sqrt{\frac{1}{h} \sum_{n=i+1}^{j}W^2_2(\rho_h^{n},\rho_h^{n+1})} 
% \\
\leq 
% &
\sqrt{(j-i)h}\sqrt{\frac{1}{h} \sum_{n=0}^{N-1}W^2_2(\rho_h^{n},\rho_h^{n+1})}
\\
\leq& \sqrt{(j-i)h}\sqrt{\frac{2}{h} \sum_{n=0}^{N-1}W^2_2(\rho_h^{n},\tilde{\rho}_h^{n+1})+W^2_2(\tilde{\rho}^{n+1},\rho^{n+1})}
% \\
\leq
% &
C \sqrt{(j-i)h} \leq C\sqrt{s-t}.
\end{align*}
% Therefore we can use Arzela-Ascoli \cite{kelley2017general} 
Therefore the family $\nu_h$ is uniformly equicontinuous, \cite{ambrosio2008gradient}*{Proposition 3.3.1}. This implies that for all $t\in[0,T]$ $\{\nu_h(t,\cdot)\}_{h>0} \subset (\cP_2(\bR^d),W_2)$ is (point-wise) precompact,  and hence we can use Arzela-Ascoli \cite{salamontheo}*{Theorem 1.1.11} to obtain uniform convergence (taking subsequences if necessary) in $W_2$ (over $[0,T]$ as $h \to 0$) of the paths $\nu_h$ to a limit, which we call $\rho$. We are then able to deduce the uniform convergence of $\rho_h$ to $\rho$ from that of $\nu_h$, namely 
\begin{align*}
   \lim_{h\to0} \sup_{t\in [0,T]} W_2(\rho_h(t),\rho(t))
   \leq
   & \lim_{h\to0} 
   \sup_{t\in [0,T]}\Big( W_2\big(\rho_h(t),\nu_h(t)\big)+W_2\big(\nu_h(t),\rho(t)\big)\Big)
    \\
    \leq& 
    \lim_{h\to0} \Big( Ch+\sup_{t\in [0,T]}W_2\big(\nu_h(t),\rho(t)\big) \Big)=0.
\end{align*}
% Hence we have that $\rho_h(t,\cdot)$ converges in $W_2$ (uniformly in $t$) to some $\rho(t,\cdot) $. % Moreover,  since the topology induced by $W_2$ is stronger than that of weak convergence we also have $\rho_h(t) \rightharpoonup \rho(t) $ in $\cP_2(\bR^d)$ for all $t\in[0,T]$. (and has a density by uniform Entropy bound and l.s.c of entropy?). 
Since $W_2$ induces a stronger topology than that of weak convergence we have for all $t\in[0,T]$ $\rho_h(t,\cdot)\rightharpoonup \rho(t,\cdot) \in \cP_2(\bR^d)$. Moreover, by the uniform moment and entropy bounds (see Lemma \ref{lemma 8}) we have by \eqref{eq appendix l.s.c of entropy} that the limit $\rho(t,\cdot) \in \cP_2^r(\bR^d)$.
By an almost identical procedure (this time concatenating geodesics between $\{\tilde{\rho}^n_h\}_n$, and using \eqref{eq 12} for $\rho^\dag_h$) we get the same convergence of $\tilde{\rho}_h,\rho^\dag_h$ to some limit % since we used the SAME concatenation 
$\tilde{\rho} \in \cP_2^r(\bR^d)$. 
It remains only to show that $\rho=\tilde{\rho}$ a.e., note we have, for instance using the Dominated Convergence theorem, letting $\varphi \in C^\infty_c([0,T),\bR^d)$

\begin{align*}
    \int_0^T\int_{\bR^d} \big(\tilde{\rho}(t,x)-\rho(t,x)\big)\varphi(t,x) dxdt=& \lim_{h\to0}\int_0^T\int_{\bR^d} \big(\tilde{\rho}_h(t,x)-\rho_h(t,x)\big)\varphi(t,x) dxdt
    \\
    =& \lim_{h\to0} \sum_{n=0}^{N-1}\int_{t_n}^{t_{n+1}}\int_{\bR^d} \big(\tilde{\rho}^{n+1}_h(x)-\rho^{n+1}_h(x)\big)\varphi(t,x) dxdt
    \\
    =& \lim_{h\to0} \sum_{n=0}^{N-1}\int_{t_n}^{t_{n+1}}\int_{\bR^{2d}} \big(\varphi(t,x)-\varphi(t,y)\big)\tilde{\gamma}^{n+1}(dx,dy)dt,
\end{align*}
where we recall $\tilde{\gamma}^{n+1}$ is the optimal coupling between $\rho^{n+1}$ and $\tilde{\rho}^{n+1}$ in $W_2$. By Taylor's theorem, Jensen inequality and then Cauchy Schwartz, we have
\begin{align*}
    \int_0^T\int_{\bR^d} \big(\tilde{\rho}(t,x)-\rho(t,x)\big)\varphi(t,x) dxdt\leq&
    \lim_{h\to0} h \sup\|\nabla \varphi\| \sum_{n=0}^{N-1}\int_{\bR^{2d}} \|x-y\| \gamma^{n+1}(dx,dy)
    \\
    \leq& 
    \lim_{h\to0} h \sup\|\nabla \varphi\| \sum_{n=0}^{N-1}W_2(\tilde{\rho}^{n+1}_h,\rho^{n+1}_h)
    \\
    \leq& \lim_{h\to0} h \sqrt{N} \sup\|\nabla \varphi\| \sqrt{\sum_{n=0}^{N-1}W^2_2(\tilde{\rho}_h^{n+1},\rho^{n+1}_h)}
    \\
   \leq&\lim_{h\to0} C h \sqrt{T} \sup\|\nabla \varphi\|=0,
    \end{align*}
    where in the last line we used Lemma \ref{lemma 5}. We are then able to conclude that  $\tilde{\rho}$ and $\rho$ are equal a.e.

\end{proof}
\end{lemma}

\textcolor{blue}{We can also ascertain the $L^1$ convergence \eqref{eq: convergence}, i.e. fix $t\in[0,T]$, we show that we have weak $L^1(\bR^d)$ convergence of $\rho_h(t,\cdot),\tilde{\rho}_h(t,\cdot),$ and $\rho^\dag_h(t,\cdot)$ to $\rho(t,\cdot)$ (the same limit as found in the previous Lemma \ref{lemma 9}), that is convergence against $L^\infty(\bR^d)$ functions not just those in $C_b(\bR^d)$. Indeed, since $x\mapsto \max\{x \log x,0\}$ is a superlinear function, the uniform bounds on the positive entropy (Lemma \ref{lemma 8}) implies the families $\{\rho_h(t,\cdot) \}_h,\{\tilde{\rho}_h(t,\cdot) \}_h,\{\rho^\dag_h(t,\cdot) \}_h$ are equi-integrable, and hence, by the weak convergence of the previous lemma,  \cite{santambrogio2015optimal}*{Box 8.2 (p301)} implies the weak $L^1(\bR^d)$ convergence. Recall this implies weak $L^1((0,T)\times \bR^d)$ convergence.}

%%%%%%%%%%%%%%%%%%%%%%%%%%%%%%%%%%%%%%%%%%%
%%%%%%%%%%%%%%%%%%%%%%%%%%%%%%%%%%%%%%%%%%%%%%%%%%%%%%%%%
%%%%%%%%%%%%%%%%%%%%%%%%%%%%%%%%%%%%%%%%%%%%%%%%%%%%%%%%%%%%%%%%%%%%

The following lemma is a key step in our analysis linking the conservative and the disspative phases together. 

\begin{lemma} For any $\varphi \in C^\infty_c([0,T)\times\bR^d)$ we have that 

\begin{align}
\nonumber
   \sum_{n=0}^{N-1} \int_{\bR^d}\Big( \tilde{\rho}^{n+1}_h(x)-\rho^{n+1}_h(x) \Big) \varphi(t_{n+1},x) dx
   =&
\int_{0}^{T}\int_{\bR^d} \rho^{\dag}_h(t,x)\big(\partial_t \varphi(t,x) + b[\rho_h(t-h)](x)\cdot \nabla \varphi (t,x) \big)dxdt  
\\
&  + \int_{\bR^d} \rho^0(x) \varphi(0,x) dx.
\label{e11}
\end{align}
\end{lemma}

\begin{proof}
Let $n\in \{0,\ldots N-1\}$. First notice that for $t\in[t_n,t_{n+1}]$ by \eqref{eq definition of the flow} and the chain rule, we have
\begin{equation}\label{eq 84}
    \partial_t \big( \varphi(t,X^n_h(t-t_{n},x))\big)=\Big(\partial_t \varphi + b[\rho_n^h] \cdot \nabla \varphi  \Big)(t,X_h^n(t-t_{n},x)).
\end{equation}
Now consider
\begin{align}
& \sum_{n=0}^{N-1} \int_{\bR^d}\Big( \tilde{\rho}^{n+1}_h(x)-\rho^{n+1}_h(x) \Big) \varphi(t_{n+1},x) dx-\int_{\bR^d} \rho^0(x) \varphi(0,x) dx
\nonumber
   \\
  & =  \sum_{n=0}^{N-1} \int_{\bR^d}\Big( \tilde{\rho}^{n+1}_h(x)\varphi(t_{n+1},x)-\rho^{n}_h(x)\varphi(t_{n},x) \Big)  dx
  \label{eq 85}
   \\
& = \sum_{n=0}^{N-1} \int_{\bR^d}\rho^{n}_h(x)\Big( \varphi(t_{n+1},X^n_h(h,x))-\varphi(t_{n},x) \Big)  dx
\nonumber
\\
& =  \sum_{n=0}^{N-1} \int_{t_n}^{t_{n+1}} \int_{\bR^d}\rho^{n}_h(x)\partial_t \big( \varphi(t,X^n_h(t-t_{n},x)) \big)  dxdt
\nonumber
\\
&= \sum_{n=0}^{N-1} \int_{t_n}^{t_{n+1}} \int_{\bR^d}\rho^{n}_h(x)\Big(\partial_t \varphi + b[\rho_h^n] \cdot \nabla \varphi  \Big)(t,X_h^n(t-t_{n},x)) dxdt
\label{eq 83}
\\
% &= \sum_{n=0}^{N-1} \int_{t_n}^{t_{n+1}} \int_{\bR^d}\rho^{\dag}_h(t,x)\Big(\partial_t \varphi + b[\rho_h^n] \cdot \nabla \varphi  \Big)(t,x) dxdt
% \\
% &= \sum_{n=0}^{N-1} \int_{t_n}^{t_{n+1}} \int_{\bR^d}\rho^{\dag}_h(t,x)\Big(\partial_t \varphi + b[\rho_h(t-h,\cdot)] \cdot \nabla \varphi  \Big)(t,x) dxdt
% \\
&=  \int_{0}^{T} \int_{\bR^d}\rho^{\dagger}_h(t,x)\Big(\partial_t \varphi + b[\rho_h(t-h,\cdot)] \cdot \nabla \varphi  \Big)(t,x) dxdt,
\label{eq 82}
\end{align}
where in \eqref{eq 85} follows since $\varphi$ has compact support, in \eqref{eq 83} we have applied \eqref{eq 84}, and in \eqref{eq 82} we have used the definitions of the interpolations $\rho_h,\rho^\dag_h$.
\end{proof}

Now following the classical procedure we can interpolate across the discrete Euler-Lagrange equations \eqref{eq euler lagrange}.
\begin{lemma}\label{lemma 20}
For any $\varphi \in C^\infty_c([0,T)\times \bR^d)$ we have
\begin{align}
\nonumber
\int_{0}^{T}\int_{\bR^d} \rho^{\dag}_h(t,x)\big(\partial_t \varphi(t,x) 
& + b[\rho(t-h,\cdot)](x)\cdot \nabla \varphi (t,x) \big)dxdt 
\\
& = - \int_{\bR^d} \rho^0(x) \varphi(0,x) dx
- h \sum_{n=0}^{N-1} \delta \cF(\rho^{n+1}_h,A_h \nabla \varphi(\textcolor{black}{t_{n+1}},\cdot))+ O(h)  .
\label{eq 15}
\end{align}

\end{lemma}
\begin{proof}
Let $n\in \{0,\ldots ,N-1\}$.
By Taylor's Theorem we have
\begin{align}
\nonumber
\int_{\bR^d}\Big( \rho^{n+1}_h(x)-\tilde{\rho}^{n+1}_h(x) \Big) \varphi( \textcolor{black}{t_{n+1}},x) dx
& 
=\int_{\bR^{2d}} \Big( \varphi( \textcolor{black}{t_{n+1}},y)-\varphi( \textcolor{black}{t_{n+1}},x) \Big) d\tilde{\gamma}^{n+1,c}_h(x,y) 
\\
&=
\int_{\bR^{2d}} \Big\langle y-x , \nabla \varphi( \textcolor{black}{t_{n+1}},y) \Big\rangle d\tilde{\gamma}^{n+1,c}_h(x,y)+\kappa_n( \textcolor{black}{t_{n+1}}).
\label{e9}
\end{align}
By Lemma \ref{lemma the cost function i.e matrix A} we can bound the remainder term $\kappa_n$, namely, 
\begin{equation}\label{e12}
    |\kappa_n(t)|\leq C\sup_{t\in[0,T),x\in \bR^d} \textcolor{black}{\|\nabla^2 \varphi \|} \int_{\bR^{2d}} \|x-y\|^2 d\tilde{\gamma}^{n+1,c}_h(x,y)\leq C\sup_{t\in[0,T),x\in \bR^d}\textcolor{black}{\|\nabla^2 \varphi \|} \int_{\bR^{2d}} c_h(x,y) d\tilde{\gamma}^{n+1,c}_h(x,y).
\end{equation}
Using \eqref{e9} in combination with the Euler-Lagrange equation \eqref{eq euler lagrange} yields the identity
\begin{align}
\nonumber
\int_{\bR^d}\Big( \rho^{n+1}_h(x)-\tilde{\rho}^{n+1}_h(x) \Big) \varphi( \textcolor{black}{t_{n+1}},x) dx=&\kappa_n( \textcolor{black}{t_{n+1}}) -h \delta \cF(\rho^{n+1}_h,A_h \nabla \varphi( \textcolor{black}{t_{n+1}}, \cdot)).
\end{align}
Summing the previous expression over $n=0,\ldots, N-1$ gives
\begin{align}
\sum_{n=0}^{N-1}\int_{\bR^d}\Big( \rho^{n+1}_h(x)-\tilde{\rho}^{n+1}_h(x)\Big) \varphi( \textcolor{black}{t_{n+1}},x) dx=& O(h) - h\sum_{n=0}^{N-1} \delta \cF(\rho^{n+1}_h,A_h \nabla \varphi( \textcolor{black}{t_{n+1}},\cdot)),
\label{e10}
\end{align}
where we have combined \eqref{e12} with Lemma \ref{lemma 5} to conclude $| \sum_{n=0}^{N-1} \kappa_n( \textcolor{black}{t_{n+1}}) | \leq C h$. Finally, using \eqref{e11} on the left hand side of \eqref{e10}, multiplying through by $-1$, delivers the sought result.
\end{proof}
We are now ready to prove the main theorem, Theorem \ref{Theorem Main}.

\begin{proof}[\textit{Proof of Theorem \ref{Theorem Main}}]
Recall the convergence result of Lemma \ref{lemma 9}. To prove Theorem \ref{Theorem Main} we need only to argue that the limit $h\to 0$, $N\to \infty$ in \eqref{eq 15} can be taken.  Clearly the error term $O(h)$ in \eqref{eq 15} goes to zero (as $h\to 0$), and for any $\varphi \in C^\infty_c([0,T),\bR^d)$ we have
$$
\lim_{h\to 0}\int_0^T\int_{\bR^d} \rho^\dag_h(t,x) \partial_t \varphi (t,x) dxdt =\int_0^T\int_{\bR^d} \rho(t,x) \partial_t \varphi (t,x) dxdt.
$$
We now address the remaining terms of \eqref{eq 15}: the free energy and the divergence free part. We start with the free energy term $\delta \cF$. 
Note that we can write 
\begin{align}
\nonumber
  h\sum_{n=0}^{N-1} & \delta \cF\big(\rho^{n+1}_h,A_h \nabla \varphi(t_{n+1},\cdot)\big) 
 \\
 &=
   \sum_{n=0}^{N-1} \int_{t_n}^{t_{n+1}} 
 \Big( \int_{\bR^d} \rho^{n+1}_h(x) \big( A_h \nabla \varphi (t_{n+1},x) \cdot  \nabla f(x) \big) dx 
 \nonumber
% \\
% &
\textcolor{black}{-} \int_{\bR^d} \rho^{n+1}_h(x) \divv \big(A_h \nabla \varphi (t_{n+1},x) \big) dx \Big)dt
\nonumber
\\
&=  \sum_{n=0}^{N-1} \int_{t_n}^{t_{n+1}} \Big( \int_{\bR^d} \rho_h(t,x)\big( A_h \nabla \varphi (t_{n+1},x) \cdot  \nabla f(x) \big)  dx 
% \nonumber
% \\
% &
\textcolor{black}{-} \int_{\bR^d} \rho_h(t,x) \divv \big(A_h \nabla \varphi (t_{n+1},x) \big) dx \Big)dt.
\label{eq 24}
\end{align}
Consider the first term on the right hand side of  \eqref{eq 24} (the second term can be dealt with in a similar manner). Adding and subtracting $A_h \nabla \varphi(t,x) $ and $A \nabla \varphi(t,x)$, we get 
\begin{align}
\nonumber
\sum_{n=0}^{N-1} 
&
\int_{t_{n}}^{t_{n+1}}\int_{\bR^d}\rho_h^{n+1}\Big( A_h \nabla \varphi (t_{n+1},x) \cdot  \nabla f(x) \Big)  dxdt
\\ \nonumber
= 
&
\sum_{n=0}^{N-1}\int_{t_{n}}^{t_{n+1}}\int_{\bR^d}\Bigg(   \rho_h \Big( A \nabla \varphi \cdot \nabla f \Big) (t,x) 
% \\
% \nonumber
% &
+ \rho_h \Big( \big(A_h-A\big) \nabla \varphi \cdot \nabla f \Big) (t,x) 
\\
\label{eq 48}
& \hspace{4cm} 
+ \rho_h(t) \Big( A_h \big( \nabla \varphi(t_{n+1}) - \nabla \varphi(t) \big) \cdot \nabla f \Big) (x) \Bigg)dxdt.
\end{align}
Then, as $h\to 0$, the first term tends to $\int_0^T\int_{\bR^d} \rho \big( A \nabla \varphi \cdot \nabla f \big) (t,x) dxdt$ by the weak  $L^1([0,T]\times \bR^d)$ convergence, the second term tends to zero by Cauchy Schwartz and the fact that $\lim_{h\to 0}\|A_h-A\|=0$ and again the weak $L^1([0,T]\times \bR^d)$ convergence. The third term in \eqref{eq 48} also tends to zero, since 
\begin{align*}
& \lim_{h\to 0}\sum_{n=0}^{N-1}\int_{t_{n}}^{t_{n+1}}\int_{\bR^d}
\rho_h(t) \Big( A_h \big( \nabla \varphi(t_{n+1}) - \nabla \varphi(t) \big) \cdot \nabla f \Big)  (x) dxdt
\\
& \leq\lim_{h\to 0} C \|A_h\| \sup_{x\in \bR^d}\|\nabla f(x)\| \sup_{[u_h,r_h]\subset [0,T), |u_h-r_h|\leq h}\sup_{s\in [u_h,r_h],x\in \bR^d }\| \nabla \varphi(r_h,x) - \nabla \varphi(s,x) \|
\int_{0}^{T}\int_{\bR^d}
\rho_h(t,x)  dxdt=0,
\end{align*}
where we have used that $\|A_h\|, \sup\|\nabla f\| \leq C$, that $\rho_h$ is a probability density, and that $\nabla \varphi$ is uniformly continuous. % (its continuous on a compact set, its compact support!)

Lastly, we address the divergence free term in \eqref{eq 15}. Adding and subtracting $\rho^\dag_h b[\rho^\dag_h]\cdot \nabla \varphi $ gives
    \begin{align}
    \nonumber
    & \int_{0}^{T}\int_{\bR^d} \rho^{\dag}_h(t,x) b[\rho_h(t-h,\cdot)](x)\cdot \nabla \varphi (t,x) dxdt 
    \\
    &= \int_{0}^{T}\int_{\bR^d} \rho^{\dag}_h(t,x) \Big(b[\rho_h(t-h,\cdot)](x)-b[\rho^\dag_h(t,\cdot)] \Big)\cdot \nabla \varphi (t,x) dxdt
    % \nonumber
    % \\
    % &
    +\int_{0}^{T}\int_{\bR^d} \rho^{\dag}_h(t,x) b[\rho^\dag_h(t,\cdot)] \cdot \nabla \varphi (t,x) dxdt. 
    \label{eq 17}
    \end{align}
The first term in \eqref{eq 17} converges to zero as $h\to 0$ since
    \begin{align}
      \big|  \int_{0}^{T}\int_{\bR^d} \rho^{\dag}_h(t,x) \Big(b[\rho_h(t-h,\cdot)](x)- & b[\rho^\dag_h(t,\cdot)](x) \Big)\cdot \nabla \varphi (t,x) dxdt \big|
      \nonumber
        \\
        \leq& C\int_0^T \Big(\int_{\bR^d} \rho^\dag_h(t,x)\|b[\rho_h(t-h,\cdot)](x)-b[\rho^\dag_h(t,\cdot)](x)\|^2dx\Big)^{\frac{1}{2}}dt
        \label{eq 18}
        \\
        \leq& C\int_0^T W_2(\rho_h(t-h,\cdot),\rho^\dag_h(t,\cdot))dt
        \label{eq 19}
        \\
        =& C\sum_{n=0}^{N-1} \int_{t_{n}}^{t_{n+1}} W_2(\rho^n_h,\rho^\dag_h(t,\cdot))dt
         \nonumber
        \\
        \leq& CT h,
        \label{eq 20}
    \end{align}
where in \eqref{eq 18} we have used the Cauchy Schwartz and Jensen's inequality, in \eqref{eq 19} we have used Assumption \eqref{eq assump 3}, and in \eqref{eq 20} we have used \eqref{eq 12} and the bounded moments result of Lemma \ref{lemma 8}. The second term on the right hand side of \eqref{eq 17} has already the desired convergence, indeed consider 
    \begin{align}
      \Big|  \int_{0}^{T}\int_{\bR^d}& \Big(\rho^{\dag}_h(t,x)  b[\rho^\dag_h(t,\cdot)](x) - \rho(t,x) b[\rho(t,\cdot)](x) \Big) \cdot \nabla \varphi (t,x) dxdt \Big| 
      \nonumber
      \\ 
      \leq& C  \int_{0}^{T} \Big(\int_{\bR^d} \rho^{\dag}_h(t,x) \| b[\rho^\dag_h(t,\cdot)](x) -  b[\rho(t,\cdot)](x) \|^2 dx \Big)^{\frac{1}{2}}dt 
      \label{eq 21}
      \\
      &\hspace{2cm}+\Big|  \int_{0}^{T}\int_{\bR^d} \big(\rho(t,x) - \rho^{\dag}_h(t,x) \big) b[\rho(t,\cdot)](x) \cdot \nabla \varphi (t,x) dxdt \Big|
      \nonumber
        \\
        \leq& 
        C T \sup_{t\in[0,T]} W_2(\rho^\dag_h(t,\cdot),\rho(t,\cdot)) +
        \Big|  \int_{0}^{T}\int_{\bR^d} \big(\rho(t,x) - \rho^{\dag}_h(t,x) \big) b[\rho(t,\cdot)](x) \cdot \nabla \varphi (t,x) dxdt \Big|,
        \label{eq 22}
    \end{align}
    where in \eqref{eq 21} we have added and subtracted $\rho^\dag_h b[\rho]$, used Cauchy Schwartz and Jensen's inequality, and in \eqref{eq 22} we used again Assumption \eqref{eq assump 3}. The two terms in \eqref{eq 22} go to zero from the convergence of $\rho^\dag_h$ in Lemma \ref{lemma 9} and that $b[\rho(t,\cdot)] \cdot \nabla \varphi \in L^\infty((0,T)\times \bR^d)$.
    
Having the above estimates, by passing to the limit $h\rightarrow 0$ in \eqref{eq 15} we obtain precisely the weak formulation \eqref{def: weak formulation general PDE} of the evolutionary equation \eqref{eq main PDE}. This completes the proof of Theorem \ref{Theorem Main}. 
\end{proof}

%
%
%
%
%%%%%%%%%%%%%%%%%%%%%%%%%%%%%%%%%%%%%%%%%%%%%%%%%%%%%%%%%%%%
%%%%%%%%%%%%%%%%%%%%%%%%%%%%%%%%%%%%%%%%%%%%%%%%%%%%%%%%%%%%
%%%%%%%%%%%%%%%%%%%%%%%%%%%%%%%%%%%%%%%%%%%%%%%%%%%%%%%%%%%%%%%
%%% BEGIN SECTION
%%%%%%%%%%%%%%%%%%%%%%%%%%%
%\section{}
%\label{sec:label}

\section{Entropy Regularised Scheme}\label{section entropy regularisation}
From a computational point of view, implementing a JKO-type scheme \eqref{eq discrete general grad flow} directly is expensive since at each iteration it requires the resolution of an optimal transportation problem. The entropic regularisation technique developed in \cite{cuturi2013sinkhorn} overcomes this difficulty by transforming the transport problem into a strictly convex problem that can be solved more efficiently with matrix scaling algorithms such as the Sinkhorn’s algorithm \cite{KnoppSinkhorn1967}. In recent years, this regularisation technique has found applications in a variety of domains such as machine learning, image processing, graphics and biology. In particular, several works have developed entropic regularisation schemes for solving evolutionary equations, such as nonlinear diffusion equations  \cites{carlier2017convergence,peyre2015entropic}, flux-limited gradient flows \cite{matthes2020discretization} and a tumour growth model of Hele-Shaw type \cite{DiMario2020}. We refer the reader to the recent monograph \cite{peyre2019computational} for a great detailed account of the entropic regularisation technique. 

In this section, we provide an entropy regularised version of the scheme introduced in Section \ref{section introduction}. The regularised scheme, presented below, differs only in that we have penalised the weighted Wasserstein distance by an entropy term. The convergence of this new scheme is stated in Theorem \ref{theorem reg-main}, the proof of which is sketched since it only differs slightly to that of  Theorem \ref{Theorem Main}. The results and techniques of this section are similar to those appearing in \cites{carlier2017convergence,adams2021entropic}. The following assumption introduces a %necessary
 theoretical constraint on the scaling of the time step and strength of entropic regularisation. It ensures that the error made by the regularisation goes to zero sufficiently fast. 

\begin{assumption}[The regularisation's scaling parameters]
\label{assumption on scaling}
Take three sequences $\{N_k\}_{k\in \bN}\subset \bN$, $\{\varepsilon_k\}_{k\in \bN}\subset \bR_+$, and $\{h_k\}_{k\in \bN} \subset \bR_+$, which, for any $k\in \bN$, abide by the following scaling

\begin{equation}\label{eq: assumption on scaling}
   h_k N_k = T,~~~\text{and}~~~ 0< \varepsilon_k \leq \varepsilon_k|\log \varepsilon_k|\leq C h_k^2,
\end{equation}
and are such that $h_k,\varepsilon_k\to 0$ and $N_k\to \infty$ as ${k\to \infty}$.
\end{assumption}

\paragraph{An entropic regularisation of the operator-splitting scheme:} Let the sequences  $\{h_k\}_{k\in \bN},\{\varepsilon_k\}_{k\in \bN},\{N_k\}_{k\in \bN}$, satisfy Assumption \ref{assumption on scaling}. \textcolor{black}{Throughout the section, for the sake of notational clarity, we have mostly suppressed the dependence of $\varepsilon,h$ and $N$ on $k$}. Let $\cF(\rho_0)<\infty$, and set $\rho^0_k=\tilde{\rho}^0_k=\rho^0$. Let $n\in\{0,\ldots,N_k-1\}$. Given $\rho^n_k$ we find $\rho^{n+1}_k$ as follows, first introduce the push forward of $\rho^n_k$ by the Hamiltonian flow as 
\begin{equation}\label{eq reg-scheme push forward}
    \tilde{\rho}^{n+1}_k=X^n_k(h,\cdot)_{\#}\rho^n_k,
\end{equation}
where $X_k^n$ solves
\begin{equation}\label{eq definition of the reg-flow}
\begin{cases}
&\partial_t X^n_k = b[\rho_k^n]\circ X^n_k,
\\
& X^n_k(0,\cdot)=\text{id}.
\end{cases}
\end{equation}
Next define $\rho^{n+1}_k$ as the minimiser of the regularised JKO type descent step
\begin{equation}
\label{eq: regularized scheme}
      \rho^{n+1}_{k}= \text{argmin}_{\rho \in \cP^r_2(\bR^d)}\big\{\frac{1}{2 h}W_{c_{h},\varepsilon}(\tilde{\rho}^{n+1}_{k},\rho)+\cF(\rho) \big\},
\end{equation}
where $ W_{c_{h},\varepsilon}$ is the regularised weighted Wasserstein 
\begin{equation}
\label{eq: regularized cost functional}
    W_{c_{h},\varepsilon}(\mu,\nu):=\inf_{\gamma\in \Pi(\mu,\nu)}\Big\{\int_{\bR^{2d}} c_{h}(x,y)d\gamma(x,y)+\varepsilon H(\gamma)\Big\},
\end{equation}
for the same cost function defined in \eqref{eq weighted wasserstein}. \textcolor{black}{Let $\tilde{\gamma}^{n,c}_k$ be the optimal plan associated to $W_{c_{h},\varepsilon}(\tilde{\rho}^{n}_{k},\rho^{n}_{k})$}, and define the interpolations $\rho_k,\tilde{\rho}_k,\rho^\dag_k$ analogously to the un-regularised case but now with respect to the new sequences $\{\rho^{n}_{k}\}_{n=0}^{N}$ and  $\{\tilde{\rho}^{n}_{k}\}_{n=0}^{N}$.

The convergence of the above entropic regularised scheme is established in the next result.
\begin{theorem}\label{theorem reg-main}
Assume that $f,b$ and $A$ satisfy Assumption \ref{assumption : main assumptionso on drift and diffusion matrix}, and let the sequences $\{h_k\}_{k\in \bN},\{\varepsilon_k\}_{k\in \bN},\{N_k\}_{k\in \bN}$ satisfy Assumption \ref{assumption on scaling}. 
Let $\rho_0\in \cP^r_2(\bR^d)$ satisfy $\cF(\rho_0)<\infty$. Let $\{\rho^{n}_{k}\}_{n=0}^{N_{k}},\{\tilde{\rho}^{n}_{k}\}_{n=0}^{N_{k}}$ to be the solution of the regularised scheme \eqref{eq reg-scheme push forward}-\eqref{eq: regularized scheme}, with interpolations  $\rho_{k},\tilde{\rho}_k,\rho^\dag_k$ as defined above. 

Then
\begin{enumerate}[label=(\roman*)]
    \item \begin{equation}\label{eq reg-convergence 1}
    \rho_{k},\tilde{\rho}_k,\rho^\dag_k
    % \rho_k 
  \underset{k\to \infty}{\longrightarrow} \rho \quad\text{in}\quad L^1((0,T)\times \bR^d).
\end{equation}
\item Moreover, there exists a map  $[0,T]\ni t \mapsto \rho(t,\cdot)$ in $\cP_2^r(\bR^d)$ such that 
\begin{equation}\label{eq reg-convergence 2}
    \sup_{t\in[0,T]} \max \Big\{\, W_2\big( \rho_{k}(t,\cdot),\rho(t,\cdot)\big),\   W_2\big(\tilde{\rho}_k(t,\cdot),\rho(t,\cdot)\big),\  W_2\big(\rho^\dag_k(t,\cdot),\rho(t,\cdot)\big)\, \Big\}
    % \rho_k 
    \underset{k\to \infty}{\longrightarrow} 0,
\end{equation}
\end{enumerate}
where the limits $\rho$ appearing above are weak solutions of the evolution equation \eqref{eq main PDE} in the  sense of Definition \ref{def: weak formulation general PDE}.
\end{theorem}

% Note Theorem \ref{theorem reg-main} only differs to Theorem \ref{Theorem Main} in that we have added regularisation to the optimal transport problem that appears in the minimisation step \eqref{eq: regularized scheme}. Hence
The proof does not change much from that of Theorem \ref{Theorem Main}, so we provide only a sketch, highlighting the parts that are different. 

\begin{proof}[Proof of Theorem \ref{theorem reg-main}]

Let $n\in\{0,\ldots,N-1\}$.

  \textit{The well-posedness.} The well-posedness of the regularised scheme relies on the well-posedness of the minimisation problem \eqref{eq: regularized cost functional}, the proof of which can be found in \cite{adams2021entropic}*{Section 3}. 

 \textit{A priori estimates.} In the proof of Theorem $\ref{Theorem Main}$ we compare the quantity $\frac{1}{2h}W_{c_{h}}(\tilde{\rho}_h^{n+1},\rho_h^{n+1})+\cF(\rho_h^{n+1})$ against $\frac{1}{2h}W_{c_h}(\tilde{\rho}_h^{n+1},\tilde{\rho}_h^{n+1})+\cF(\tilde{\rho}_h^{n+1})$. 
The term $W_{c_h}(\tilde{\rho}_h^{n+1},\tilde{\rho}_h^{n+1})$ is zero, and hence we end up with a control of $W_{c_{h}}(\tilde{\rho}_h^{n+1},\rho^{n+1}_h)$ in terms of the free energy. 
However, since $W_{c_{h},\varepsilon}(\tilde{\rho}^{n+1}_k,\tilde{\rho}^{n+1}_k)\neq 0$, we need to select a new distribution to compare the performance of $\rho^{n+1}_k$ against. 
We judiciously choose a distribution $\rho_\epsilon$ (with optimal plan $\gamma_\epsilon)$ as to make the cost of transporting mass \textcolor{black}{zero, i.e as $\epsilon\to 0$ we aim to have $(c_h,\gamma_\epsilon)\to 0$}.  We construct such a candidate distribution $\rho_\varepsilon$ in the following way, let $G\in C^\infty_c (\mathbb{R}^d)$ be a probability density, such that $M(G)=1$ and $H(G) <\infty$. Define $G_\varepsilon(\cdot):=\varepsilon^{-2d}G(\frac{\cdot}{\varepsilon^2})$, and 
\begin{equation*}
   \gamma_\varepsilon(x,y) 
   := 
   \tilde{\rho}^{n+1}_k(x) G_\varepsilon \big(y-x\big),
\end{equation*}
as the joint distribution with first marginal $\tilde{\rho}^{n+1}_k$, and second marginal 
% \begin{equation*}
$\rho_\varepsilon(y) :=\int \gamma_\varepsilon(x,y) dx$. 
% \end{equation*}
One can then calculate/express  $H(\gamma_\varepsilon),\cF(\gamma_\varepsilon),(c_{h},\gamma_\varepsilon)$ in terms of $\tilde{\rho}^{n+1}_k$ (see \cite{adams2021entropic}*{Lemma 4.3}). Comparing the performance of $\rho^{n+1}_k$ against $\rho_\varepsilon$ in \eqref{eq: regularized scheme} we get, making use of the scaling \eqref{eq: assumption on scaling} and that $H(\tilde{\gamma}^{n+1,c}_k)\geq H(\rho^{n+1}_k)+H(\tilde{\rho}^{n+1}_k)$, the following inequality
\begin{equation}\label{eq 27}
    (c_h,\tilde{\gamma}^{n+1,c}_k) \leq Ch^2\Big(M(\tilde{\rho}^{n+1}_k)+1\Big)-\varepsilon H(\rho^{n+1}_k)+2h\Big( \cF(\tilde{\rho}^{n+1}_k)- \cF(\rho^{n+1}_k) \Big).
\end{equation}
We are able to obtain bounded 2nd moments, energy, and entropy estimates in an almost identical fashion as to Lemma \ref{lemma 3}, specifically in \eqref{z13} we use $(c_h,\tilde{\gamma}^{i+1,c}_k)$ in place of $W_{c_h}(\tilde{\rho}^{i+1},\rho^{i+1})$, and apply \eqref{eq 27}. Moreover, summing \eqref{eq 27} and using such estimates yields the bound 
\begin{equation}\label{eq 28}
    \sum_{n=0}^{N-1}(c_h,\tilde{\gamma}^{n+1,c}_k) \leq C h.
\end{equation}
It is easy to conclude that we also have 
\begin{equation}
     \sum_{n=0}^{N-1} W_2^2(\tilde{\rho}^{n+1}_k,\rho^{n+1}_k)   \leq C h
     \quad \textrm{and}\quad 
     \sum_{n=0}^{N-1} W_2^2(\rho^{n}_k,\tilde{\rho}^{n+1}_k)   \leq C h.
\end{equation}

 \textit{The Discrete Euler-Lagrange Equation and concluding the convergence.} 
Since $\rho^{n+1}_k$ solves the minimisation problem \eqref{eq: regularized scheme}, the associated discrete Euler-Lagrange equation reads, for any  $\varphi \in C_c^\infty(\bR^d)$, 
\begin{equation}
  0=\frac{1}{h}\int_{\bR^{2d}} \Big\langle x-y, \nabla\varphi(y)  \Big\rangle  d\tilde{\gamma}^{n+1,c}(x,y) +\delta \cF(\rho^{n+1}_k,A_h \nabla \varphi)-\frac{\varepsilon}{2h}\int_{\bR^d}\rho^{n+1}_k(y)\divv\big(A_h \nabla \varphi(y)  \big)dy.
  \label{eq reg-euler lagrange}
\end{equation}
Therefore, the analogous result to Lemma \ref{lemma 20} is
\begin{align}
& \int_{0}^{T}\int_{\bR^d} \rho^{\dag}_k(t,x)\big(\partial_t \varphi(t,x) +  b[\rho(t,\cdot)](x)\cdot \nabla \varphi (t,x) \big)dxdt  \nonumber
\\& = - \int_{\bR^d} \rho^0(x) \varphi(0,x) dx
- \sum_{n=0}^{N-1} \Big(  h \delta \cF(\rho^{n+1}_k,A_h \nabla \varphi(t_n,\cdot)) +\frac{\varepsilon}{2h}\int_{t_n}^{t_{n+1}}\int_{\bR^d}\rho^{n+1}_k(y)\divv\big(A_h \nabla \varphi(t_{n},y) \big)dy   \Big) +O(h).
\label{eq 31}
\end{align}
The convergence claimed in \eqref{eq reg-convergence 1} and \eqref{eq reg-convergence 2} follows by a priori estimates identical to those of Lemma \ref{lemma 9}. Hence to complete the proof of Theorem \ref{theorem reg-main} we need only to deal with the term appearing from the regularisation 
\begin{align*}
    \sum_{n=0}^{N-1}  \frac{\varepsilon}{2h} 
    \int_{t_{n}}^{t_{n+1}}\int_{\bR^d}\rho^{n+1}_k(y)\divv\big(A_h \nabla \varphi(t_{n},y) \big)dy,
\end{align*}
and show it goes to zero as $\varepsilon,h \to 0$. This is clear by a similar argument to that in the end of the proof of Section \ref{section convergence}: using the convergence \eqref{eq reg-convergence 1} of $\rho_k$ and the scaling \eqref{eq: assumption on scaling} which implies that $\frac{\varepsilon}{h}\to 0$.

\end{proof}

%
%
%
%
%%%%%%%%%%%%%%%%%%%%%%%%%%%%%%%%%%%%%%%%%%%%%%%%%%%%%%%%%%%%
%%%%%%%%%%%%%%%%%%%%%%%%%%%%%%%%%%%%%%%%%%%%%%%%%%%%%%%%%%%%
%%%%%%%%%%%%%%%%%%%%%%%%%%%%%%%%%%%%%%%%%%%%%%%%%%%%%%%%%%%%%%%
%%% BEGIN SECTION
%%%%%%%%%%%%%%%%%%%%%%%%%%%
%\section{}
%\label{sec:label}

\section{Examples}\label{section examples}
In this section, we present four concrete examples of evolutionary equations that can all be written in the general form \eqref{eq main PDE}: the Vlasov-Fokker-Planck equation,  a degenerate non-linear diffusion equation of Kolmogorov-type, the linear Wigner FPE, the regularised Vlasov-Poisson FPE, and a generalised Vlasov-Langevin equation. 
\medskip

\textbf{Applicability of Theorems \ref{Theorem Main} and  \ref{theorem reg-main} to the examples.} In all these examples, we will show explicitly the (non-local) vector field $b$ and the diffusion matrix $A$. Assuming the drift vector fields and the diffusion matrix are such that Assumption \ref{assumption : main assumptionso on drift and diffusion matrix} is satisfied, then Theorem \ref{Theorem Main} and/or of Theorem \ref{theorem reg-main} provides novel operator-splitting variational schemes for solving these evolutionary equations. It will be clear that from the explicit formulas that $A$ is symmetric positive semi-definite and $b[\rho]$ is divergence-free. It remains to consider the first and second conditions in \eqref{eq assump 4}, which are assumptions on the growth and regularity of the vector field, and \eqref{eq assump 3}. In all examples, the vector field $b[\rho](x)$ consists of a local part and a non-local part, where the non-local part is a convolution of $\rho$ with an interaction kernel $K$, namely of the form

$$
b_{\text{non-local}}[\rho](x):=(K * \rho)(x)=\int K(x-x')d\rho(x').
$$
Assumption \ref{assumption : main assumptionso on drift and diffusion matrix} is satisfied for instance when the local part is Lipschitz, and the kernel $K$ is uniformly bounded, Lipschitz, and differentiable. Under this assumption, \eqref{eq assump 4} and \eqref{eq assump 3} are straightforward for the local part. Now we show that the non-local part also fulfills these assumptions \textcolor{black}{if it is uniformly bounded and Lipschitz.}
\begin{lemma}\label{lemma lipschitz kernal}
Suppose that $K$ is \textcolor{black}{uniformly bounded} and Lipschitz. Then for all $\rho,\mu\in \mathcal{P}_2(\bR^d),~ z \in \mathbb{R}^d$, we have
\begin{align}
& \int_{\bR^d} \| K*\rho(z) - K*\mu(z)\|^2 ~ d\rho(z) \leq C W_2^2(\rho, \mu), \label{assump 5}
\\ & \| K*\mu(z) \|\leq C(1+\|z\|),\label{assump 6}
\\& K*\mu\in W^{1,1}_{\text{loc}}(\bR^d)\label{assump 7}.
\end{align}
\end{lemma}
\begin{proof}
We first prove \eqref{assump 5}. We will use the following equivalent formulation of the Wasserstein distance \cite{villani2021topics}
\begin{equation}
\label{W2 def2}
    W_2^2(\rho,\mu)=\inf\Big[\mathbb{E}(\|X-Y\|^2)\Big],
\end{equation}
where the infimum is taken over all couples of random variables $X$ and $Y$ with $Y\sim \rho$ and $X\sim \mu$.

Now let $\mu,\rho \in \cP_2(\bR^d)$ and take random variables $X$ and $Y$ with $Y\sim \rho$ and $X\sim \mu$. We have
\begin{align*}
\int_{\bR^d} \| K*\rho(z) - K*\mu(z)\|^2 ~ d\rho(z) =& \int_{\bR^d} \| \int_{\bR^d}K(z-z')\big(d\rho(z')-d\mu(z')\big)\|^2 ~ d\rho(z) 
\\
=& \int_{\bR^d} \| \mathbb{E}[K(z-Y)-K(z-X)] \|^2 ~ d\rho(z) 
\\
\leq&\int_{\bR^d}\mathbb{E}[\|K(z-Y)-K(z-X)\|^2]  ~ d\rho(z)
\\
\leq&C \int_{\bR^d}\mathbb{E}[\|Y-X\|^2]  ~ d\rho(z)=C \mathbb{E}[\|Y-X\|^2].
\end{align*}
Taking infimum over all $X$ and $Y$ and using \eqref{W2 def2} yields \eqref{assump 5}. \textcolor{black}{Verifying \eqref{assump 6} is straightforward by the uniform bound on $K$.} Finally, we check \eqref{assump 7}. Let $\Omega$ be an arbitrary compact set in $\bR^d$. Firstly it is clear that $ K*\mu(z)\in L^1_{\text{loc}}(\bR^d)$ since $K$ is uniformly bounded.
 Let $i,j\in {1,\ldots,d}$. It remains to show $\partial_{z_{j}} K_i*\mu(z)\in L^1_{\text{loc}}(\bR^d)$. In fact, since $\| \nabla K_i \|\leq C$, we have
 
\begin{align*}
    \int_{\Omega} \|  \int \partial_{z_{j}} K_i(z-z') d\mu(z') \| dz \leq& \int_{\Omega} \|  \int \nabla K_i(z-z') d\mu(z') \| dz \leq  C |\Omega|.
\end{align*}
This completes the proof of this lemma.
\end{proof}
We now discuss concrete applications of our work.
\subsection{Linear Wigner FPE}\label{sec example linear wigner}
The Wigner Fokker-Planck equation is the quantum mechanical analogue to the classical VFPE discussed in Section \ref{section examples kfpe}. It has been used in the modelling of semiconductor devices, see \cite{markowich2012semiconductor} and references therein. In particular, the linearized Wigner Fokker-Planck equation is
\begin{align}\label{eq linear WIGNER FPE}
\partial_t \rho
=&
\textcolor{black}{- v\cdot \nabla_x \rho }+ \beta \divv_v\big( v \rho \big) + \sigma \Delta_v \rho  + \alpha \Delta_x \rho + 2\lambda \divv_v(\nabla_x \rho),
\end{align}
where the variables $x,v\in\bR^d$, and the constant $\beta>0$ is the friction parameter, and $\sigma,\alpha,\lambda>0$ form the diffusion matrix of the system. The above equation is an instance of \eqref{eq main PDE} with
$$
b(x,v)=\begin{pmatrix}
    \textcolor{black}{\big(\frac{\beta \lambda}{\sigma}+1\big)v}
    \\
  0
    \end{pmatrix},
\qquad
A=
\begin{pmatrix}
\alpha & \lambda
\\
\lambda & \sigma
\end{pmatrix},
\qquad
f(x,v)=f(v)=\frac{\beta}{\textcolor{black}{2}\sigma} \|v\|^2,
$$
Note that each entry of $A$ stands for a $d\times d$ identity matrix times that entry. Note that the above equation is local and non-degenerate. In the operator-splitting scheme, one needs not to perturb the diffusion matrix $A$ (see discussion at the end of the introduction of the scheme in Section \ref{sec the scheme}).

It should be mentioned that the full Wigner Fokker-Planck equation includes a pseudo-differential operator acting on a non-local term coupled to the Poisson equation. At the moment, it is not clear to us whether we can extend our results to the full equation.

\subsection{Vlasov-Fokker-Planck Equation (VFPE)}\label{section examples kfpe}
The Vlasov-Fokker-Planck equation, which describes the probability of finding a particle at time $t$ with position $x\in \bR^d$ and velocity $v\in\bR^d$ moving under the influence of an external potential $\nabla g$, an interaction force $K$, a frictional force $\nabla f$ and a stochastic noise, is given by
\begin{equation*}
    \partial_t \rho = \textcolor{black}{-v\cdot\nabla_x\rho +\nabla g \cdot \nabla_v \rho} +\divv_v\big(\rho  K\ast\rho \big)+\divv_v \big( \rho  \nabla f \big) + \Delta_v \rho.
\end{equation*}
It is the forward Kolmogorov equation of the follow stochastic differential equation
\begin{equation*}
 \begin{cases}
 &dX_t=V_tdt
    \\
    &dV_t=-\big(K*\rho_t(X_t)+\nabla g(X_t)\big)dt-\nabla f(V_t)dt+\sqrt{2}dW_t 
    \\
    &\rho_t=\text{Law}(X_t,V_t).
 \end{cases} 
\end{equation*}
The VFPE is a special case of \eqref{eq main PDE} with
\begin{equation}\label{eq b and A for KFPE}
b[\rho](x,v)=\begin{pmatrix}
    v
    \\
   - \big(\nabla g(x) + K * \rho (x) \big)
    \end{pmatrix},
    \qquad 
A=
\begin{pmatrix}
0 & 0 
\\
0 & I
\end{pmatrix},  
\quad f(x,v)=f(v),\quad (x,v)\in\bR^{2d}.
\end{equation}
When there is no interaction (i.e., $K=0$) the VFPE reduces to the Kramers equation. As mentioned in the introduction, various variational schemes have been developed for the Kramers equation \cites{duong2014conservative, Huang00,carlen2004solution,marcos2020solutions}, see \cite{adams2021entropic} for extensions of these work to non-linear models. This paper not only provides a novel scheme but also incorporates the interaction force.

\subsection{Regularized Vlasov-Poisson-Fokker-Planck equation}
\label{sec reg VPFPE}
The Vlasov-Poisson-Fokker-Planck equation is given by
\begin{equation}\label{eq vlasov-poisson-fokker-planck}
\partial_t \rho = \textcolor{black}{-v\cdot\nabla_x\rho+\nabla \big( g (x)+\phi[\rho](x)\big)\cdot \nabla_v \rho } -\beta\divv_v\big( \rho v \big) + \sigma \Delta_v \rho.
\end{equation}
for positive constants $\sigma,\beta$ and variables $x,v\in\bR^d$, where $\phi$ solves the Poisson equation 
\begin{equation*}
    \Delta \phi (x)=- \int_{\bR^d} \rho(x,v)dv,
\end{equation*}
the solution of which is 
\begin{equation}\label{VPFPE the non-local term}
    \phi(x)= \int_{\bR^{2d}} \Gamma(x-y)\rho(y,v)dydv, 
\end{equation}
for $\Gamma$ defined as 
\begin{equation*}
    \Gamma(r):= 
    \begin{cases}
   \frac{\omega_d}{\|r\|^{\frac{d-2}{2}}} & \text{for}~d>2,
    \\
  \omega_2\log\|r\|  &\text{for}~d=2,
    \end{cases}
\end{equation*}
where $\omega_d$ is the surface area of the unit ball in $\bR^d$. 
% Coulomb's law  The law states that the magnitude of the electrostatic force of attraction or repulsion between two point charges is directly proportional to the product of the magnitudes of charges and inversely proportional to the square of the distance between them
%  
This equation is of great importance in plasma physics, as it models a cloud of charged particles influencing each other through a Coulomb interaction, whilst subject to deterministic and random forcing, and friction. Since $\Gamma$ is singular our methods can not be directly applied where it is easy to check that \eqref{eq assump 4} fails to hold. However, if we consider $\phi^\epsilon$ defined analogously to \eqref{VPFPE the non-local term} but with $\Gamma$ replaced by 
\begin{equation*}
    \Gamma^\epsilon(r)= 
    \begin{cases}
   \frac{\omega_d}{\big(\|r\|^2+\epsilon\big)^{\frac{d-2}{2}}} & \text{for}~d>2
    \\
   \frac{\omega_d}{2}\log \big(\|r\|^2+\epsilon \big) &\text{for}~d=2
    \end{cases},
\end{equation*}
then we arrive at the regularised Vlasov-Poisson Fokker-Planck equation
\begin{equation}\label{eq reg vlasov-poisson-fokker-planck}
\partial_t \rho^\epsilon = \textcolor{black}{-v\cdot\nabla_x\rho^\epsilon+\nabla \big( g (x)+\phi^\epsilon[\rho^\epsilon](x)\big)\cdot \nabla_v \rho^\epsilon } -\beta\divv_v\big( \rho^\epsilon v \big) + \sigma \Delta_v \rho^\epsilon.
\end{equation}

Here we have regularised the Kernel appearing in the convolution (this is different from the regularisation discussed in Section \ref{section entropy regularisation}). For any $\epsilon>0$, $\Gamma^\epsilon$ is no longer singular, and  \textcolor{black}{$\|\nabla \Gamma^\epsilon\|$ is uniformly bounded.} Moreover, $\nabla \Gamma^\epsilon$ is Lipschitz, indeed, the Hessian is uniformly bounded which can be seen from the following explicit computations, for $d\geq 2$ we have
\begin{equation*}
    \partial_{x_i} \partial_{x_j} \Gamma^{\epsilon} (x) = C_d \begin{cases}
    \frac{1}{(\|x\|^2+\epsilon)^{\frac{d}{2}}}-\frac{dx_i^2}{(\|x\|^2+\epsilon)^\frac{d+2}{2}} & ~\text{if}~i=j,
    \\
   -\frac{dx_ix_j}{(\|x\|^2+\epsilon)^{\frac{d+2}{2}}} &~\text{if}~i \neq j,
    \end{cases} 
\end{equation*}
for some constant $C_d$ depending only on the dimension. Hence by Lemma \ref{lemma lipschitz kernal} assumptions \eqref{eq assump 4} and \eqref{eq assump 3} are satisfied.

The solutions $\rho^\epsilon$ to \eqref{eq reg vlasov-poisson-fokker-planck} have been shown to converge (as $\epsilon\to 0$) to the solution of the original system \eqref{eq vlasov-poisson-fokker-planck} \cite{carrillo1995initial}. One-step variational schemes (in the space of probability measures) have already been proposed for \eqref{eq reg vlasov-poisson-fokker-planck}, see \cite{huang2000variational}. However, the cost function used in \cite{huang2000variational} is not a metric, the free energy depends on the time step \textit{and}  contains a mix of conservative and dissipative terms. Our approach in the present paper naturally splits the conservative and dissipative dynamics.

\subsection{A generalised Vlasov-Langevin equation}

Next we consider the following generalised Vlasov-Langevin equation \cites{ottobre2011asymptotic,Duong2015NonlinearAnalysis}
\begin{equation}
\label{eq: GLE}
\partial_t\rho =-p\cdot\nabla_q\rho +  \big(\cA(q)+K\ast \rho (q)-\sum_{j=1}^m\Lambda^j z^j  \big) \cdot \nabla_p \rho+\sum_{j=1}^m\mathrm{div}_{z^j}[(\Lambda^j p +\alpha^j\, z^j)\rho ]+\Delta_z\rho_t.%,\quad\text{in}.%~~\R^{3d}\times \R_+.
\end{equation}
Note that in the above equation, the coordinates are $(q,p,z)\in \bR^{2d+md}$, with $q,p \in \bR^{d}$ and $z\in \bR^{md}$ for some $m \in \bN$. Equation \eqref{eq: GLE} is the forward Kolmogorov equation of the SDE system 
\begin{equation}
\label{eq: SDE of GVD}
\begin{cases}
dQ_t=P_t\,dt,
\\dP_t=-\cA(Q_t)\,dt-K\ast\rho_t(Q_t)\,dt+\sum_{j=1}^m\Lambda^j\,Z_t^j\,dt,
\\dZ^j_t=-\Lambda^j\,P_t\,dt-\alpha^j\, Z^j_t\,dt+\sqrt{2}\,dW^j_t,~~~~~j=1,\ldots,m.
 \\
    \rho_t=\text{Law}(Q_t,P_t,Z_t^1,\ldots,Z_t^m),
\end{cases}
\end{equation}  
where $W^j_t$ are independent $d-$dimensional Brownian motions, $\cA:\bR^{d}\to\bR^{d}$ is an external potential, $K:\bR^{d}\to\bR^{d}$ an interaction kernel, $\Lambda^j,\alpha^j \in \bR^{d\times d}$ constant diagonal matrices $\forall j\in\{1,\ldots,m\}$. When $\cA$ is the gradient of a potential and no Kernel is present, $K=0$, then \eqref{eq: SDE of GVD} can be viewed as the coupling of a deterministic Hamiltonian system $(Q_t,P_t)$ to a heat bath $Z_t$, the literature on this subject is vast. In this setup, for large $m$ the Markovian system \eqref{eq: SDE of GVD} approximates the Generalised Langevin equation (GLE). The GLE serves as a standard model in non-Markovian non-equilibrium statistical mechanics, where the Hamiltonian system is in contact with one or more heat baths. The heat baths are modeled by the linear wave equation and are initialised according to Gibbs distribution, see  \cites{kupferman2004fractional,ottobre2011asymptotic,rey2006open} and references therein. When $K\neq 0$, the mean field term $K*\rho$ models the particle interactions in the underlying deterministic system (via the positions $Q_t$). In this case, \eqref{eq: SDE of GVD} is the McKean-Vlasov limit of a system of weakly interacting particles \cites{Duong2015NonlinearAnalysis}. 

Again the generalised Vlasov-Langevin equation is another example of \eqref{eq main PDE} where  free vector field, diffusion matrix, and potential energy are given by
\begin{align*}
    b[\rho](q,p,z)=\begin{pmatrix}
    p\\ - \cA(q)-K\ast\rho(q)+\sum_{j=1}^m\Lambda^j z^j
    \\ 
    -\Lambda^1 p
    \\
    \vdots
    \\
    -\Lambda^m p
    \end{pmatrix}, \quad A=\begin{pmatrix}
    0&0&0\\
    0&0&0\\
    0&0& I
    \end{pmatrix}, \quad f(q,p,z)=f(z)=\sum_{j=1}^m\frac{1}{2}\|\alpha^j \textcolor{black}{z^j}\|^2, 
\end{align*}

where $I$ is the $md\times md$ identity matrix. 
\subsection{A degenerate diffusion equation of Kolmogorov-type}
The final example that we consider is the following non-linear degenerate  equation of Kolmogorov type
\begin{equation}
    \label{eq: gKramers}
\partial_t\rho = -\sum_{i=2}^{n}\textcolor{black}{x_i \cdot \nabla_{x_{i-1}} \rho}+\divv_{x_n}(\nabla f(x_n)\rho)+\Delta_{x_n} \rho.
\end{equation}
In the above equation, the coordinates are $\x=\begin{pmatrix}
    x_1,
    x_2,
    \ldots,
    x_{n-1},
    x_n
    \end{pmatrix}^T$, where $x_i\in \bR^{d}$ for each $i\in\{1,\ldots ,n\}$.
Equation \eqref{eq: gKramers} is the forward Kolmogorov equation of the associated stochastic differential equations
\begin{equation}\label{eq: SDE1}
\begin{cases}
&dX_1=X_2\,dt
\\&dX_2=X_3\,dt
\\&\quad\vdots
\\&dX_{n-1}=X_{n}\,dt
\\&dX_n=-\nabla f(X_n)\,dt +\sqrt{2}\, dW_t,
\end{cases}
\end{equation}
where $W_t$ is a $d$-dimensional Wiener process. System \eqref{eq: SDE1} describes the motion of $n$ coupled oscillators connected to their nearest neighbours with the last oscillator additionally forced by a random noise which propagates through the system. The simplest cases of $n=1,n=2$ correspond to the heat equation and Kramers' equation (with no background potential) respectively.  When $n>2$ this type of equations arise as models of simplified finite Markovian approximations of generalised Langevin dynamics \cite{ottobre2011asymptotic}, % 
or harmonic oscillator chains \cites{bodineau2008large,delarue2010density}. Recent works \cites{DuongTran17,adams2021entropic} have constructed a one-step scheme for \eqref{eq: gKramers}, however, the cost function used there (the mean squared derivative cost function \cite{DuongTran18}[(11)]), although explicit, does not take a simple form.

Equation \eqref{eq: gKramers}  is yet another special case of \eqref{eq main PDE} with the following divergence free vector field, diffusion matrix, and potential energy
\begin{equation}\label{eq coef of degenerate diffusion}
b(\x)=(
    x_2
    ,
    x_3
    ,
    \ldots
    ,
    x_n
 ,
    0
    )^T,
~~~~
~~~~
A=
\begin{pmatrix}
0 & 0
\\
0 & I
\end{pmatrix},
~~~~
~~~~
f(\x)=f(x_n),
\end{equation}
where, in the matrix $A$, $I$ is the $d\times d$-dimensional identity matrix, and remaining elements are all $0$.
%%%%%%%%%%%%%%%%%%%%%%%%%%%%%%%%%%%%%%%%%%%%%%%%%%%%%%%%%%%%
%%%%%%%%%%%%%%%%%%%%%%%%%%%%%%%%%%%%%%%%%%%%%%%%%%%%%%%%%%%%
%%% BEGIN APPENDIX SECTION
%%%%%%%%%%%%%%%%%%%%%%%%%%%
% \newpage
\appendix

\section{Appendix}

In this section, we provide detailed computations and proof for some technical results used in previous sections for the completeness.
\begin{lemma}\label{lemma 1}
For any $h>0$, and any $\mu$ and $\nu$ in $\cP_2(\bR^d)$, it is true that 

\begin{equation}\label{eq moments and wasserstein}
M(\nu)\leq 2\big(W_2^2(\mu,\nu)+M(\mu)\big)
\end{equation}

and 

\begin{equation}\label{eq moments and weighted wasserstein}
   M(\nu)\leq C\big(W_{c_{h}}(\mu,\nu)+M(\mu)\big).
\end{equation}
 \begin{proof}
 The result \eqref{eq moments and wasserstein} for $W_2$ is obvious. For \eqref{eq moments and weighted wasserstein} just use \eqref{eq moments and wasserstein} in conjunction with \eqref{eq 2}. 
 \end{proof}
\end{lemma}

\begin{lemma}\cite{jordan1998variational}*{Proposition 4.1}\label{lemma appendix 1}  
There exists a $C>0$ and $0<\alpha<1$ such that 
\begin{equation}\label{eq appendix 1}
    H(\mu)\geq-C(M(\mu)+1)^{\alpha},~~\forall \mu\in \cP_2^r(\bR^d)
    \quad\textrm{and}\quad
   \textcolor{black}{ H_{-}}(\mu)\leq C(M(\mu)+1)^{\alpha},~~\forall \mu\in \cP_2^r(\bR^d).
\end{equation}

Moreover, $H$ is  weakly lower semi-continuous under bounded moments, i.e., if $\{\mu_k\}_{k\in \bN}\subset \cP_2(\bR^d)$, $\mu \in \cP_2(\bR^d)$ with $\mu_k\rightharpoonup\mu$, and there exists $C>0$ such that $M(\mu_k),M(\mu)< C$ for all $k\in \bN$, then
\begin{equation}\label{eq appendix l.s.c of entropy}
    H(\mu)\leq \liminf_{k\to\infty}H(\mu_k).
\end{equation}  
\end{lemma}

\begin{lemma}\label{lemma 25}
Let $h>0$. Given $\mu,\nu \in \cP_2^r(\bR^d)$ there exists $\gamma\in \Pi(\mu,\nu)$ such that 
$$ W_{c_{h}}(\mu,\nu)=(c_h,\gamma).$$

Moreover, the map $\gamma \mapsto (c_h,\gamma)$ is weakly lower semi-continuous.
\begin{proof}
% Since $c_h$ is continuous and non-negative, the map $ \cP(\bR^{2d}) \ni \gamma\mapsto (c_h,\gamma)$ is weakly lower semi-continuous by \cite[Lemma 4.3]{villani2008optimal}. 
See \cite{villani2008optimal}*{Theorem 4.1}.
\end{proof}
\end{lemma}

\begin{lemma}[Lower Semi-Continuity of the functionals]\label{lemma 27}
Let $\{\nu_k\}_{k\in\bN}\subset \cP_2^r(\bR^d)$, $\mu,\nu \in \cP_2^r(\bR^d)$, with $\nu_k \rightharpoonup \nu$ as $k\to\infty$.  Assume that for all $k\in \bN$ the probability measures $\nu_k,\mu,\nu $ have uniformly bounded entropy and second moments. Then
\begin{equation}
   \cF(\nu) \leq \liminf_{k\to\infty} \cF(\nu_k)
   \quad \textrm{and}\quad 
   W_{c_{h}}(\mu,\nu) \leq \liminf_{k\to \infty} W_{c_{h}}(\mu,\nu_k).
\end{equation}

\begin{proof}
Let $\{\nu_k\},\mu,\nu$ be as assumed above, and $\{\gamma_k\}$ be the associated optimal plans in $W_{c_{h},\varepsilon}(\mu,\nu_k)$. 
Note $\{\gamma_k\}\subset \Pi(\mu,\{\nu_k\})$ (see notations Section \ref{sec: notation}). 
Since $\{\nu_k\}$ is weakly convergent then it is tight, and \cite{villani2008optimal}*{Lemma 4.4} implies that $\Pi(\mu,\{\nu_k\}) $ is so too. Extracting (and relabelling) a subsequence $\{\gamma_k\}$, we know that (as $k\to\infty$) $\gamma_k\rightharpoonup \gamma \in \cP(\bR^{2d})$. 
In fact $\gamma \in \Pi(\mu,\nu)$ since the weak convergence of $\gamma_k$ implies the weak convergence of its marginals (and we know $\nu_k \rightharpoonup \nu$). 
Now, the lower semi-continuity described in Lemma \ref{lemma 25} implies that
\begin{align*}
  \liminf_{k\to \infty} W_{c_{h}}(\mu,\nu_k)=
%   &
  \liminf_{k\to \infty}
  \frac{1}{2h}(c_h,\gamma_k)
%   \\
   &
  \geq \frac{1}{2h}(c_h,\gamma)
%   \\
%   &
  \geq W_{c_{h}}(\mu,\nu).
 \end{align*} 
The proof for $\cF$ is identical to \cite{adams2021entropic}*{Lemma 3.8}.
\end{proof}
\end{lemma}

\subsection{Well-Posedness}
\label{sec:appendix:wellposedness}

\begin{proof}[Proof of Proposition \ref{lemma : well-posedness of jko scheme}]%[Proof of Proposition \ref{lemma : well-posedness of jko scheme}]
Let $0<h<1$ and $\mu,\nu \in \cP_2^r(\bR^d)$.
Define $J_{c_{h}}(\mu,\nu):=\frac{1}{2h}W_{c_{h}}(\mu,\nu)+\cF(\nu)$, then we have 
\begin{align}
    J_{c_{h}}(\mu,\nu) = \frac{1}{2h}W_{c_{h}}(\mu,\nu)+M(\mu)+\cF(\nu)-M(\mu)
    % \nonumber
    % \\
    \geq& W_{c_{h}}(\mu,\nu)+M(\mu)+\cF(\nu)-M(\mu)
    \label{eq 39}
    \\
   \label{eq 40}
    \geq& C_1 M(\nu) + H(\nu) -M(\mu)
    \\
    \label{eq 49}
    \geq& C_1 M(\nu) - C_2(1+M(\nu))^{\alpha}+C_\mu,
\end{align}
where in \eqref{eq 39} we have used that $h<1$, in \eqref{eq 40} we used Lemma \ref{lemma 1} and the non-negativity of $f$, and in \eqref{eq 49} we used Lemma \ref{lemma appendix 1}.  
We emphasize that the constants $C_1,C_2>0$ are independent of $\mu,\nu$ and $C_\mu>0$ is independent of $\nu$. 
Inequality \eqref{eq 49} implies that $\nu \mapsto J(\mu,\nu) $ is bounded from below. 
Note that there exists a $\nu\in\cP_2^r(\bR^d)$ such that  $J_{c_{h}}(\mu,\nu)<\infty$, for example,  take $\nu=\mu$ (and the product plan). 

Let \textcolor{black}{$\{\nu_k\}$} be a minimising sequence and note that this implies $M(\nu_k),H(\nu_k)$ are uniformly bounded. The uniform boundedness of $M(\nu_k)$ implies tightness of $\{\nu_k\}$, and hence extracting a subsequence we have $\nu_k \rightharpoonup \nu^* \in \cP(\bR^d)$. Moreover, $\nu^* \in \cP_2(\bR^d)$ since uniformly bounded 2nd moments and weak convergence of $\{\nu_k\}$ implies that the limit has a bounded 2nd moment as well. Additionally, $\nu^* \in \cP_2^r(\bR^d)$ by the lower semi-continuity of $H$, see Lemma \ref{lemma appendix 1}. That $\nu^*$ is indeed the minimiser of \eqref{eq the minimiser} follows from the lower semi-continuity in Lemma \ref{lemma 27}. Finally, the uniqueness follows by linearity of $F(\cdot)$, convexity of $W_2(\mu,\cdot)$, and the strict convexity of $H(\cdot)$. 
\end{proof}
\section*{Acknowledgment} We would like to thank the anonymous  referees for useful comments and suggestions which have helped us to improve the presentation of the article.
%%%%%%%%%%%%%%%%%%%%%%%%%%%%%%%%%%%%%%%%%%%%%%%%%%%%%%%%%%%%%
%%%% \END APPENDIX
%%%%%%%%%%%%%%%%%%%%%%%%%%%%%%%%%%%%%%%%%%%%%%%%%%%%%%%%%%%%%
%%%%%%%%%%%%%%%%%%%%%%%%%%%%%%%%%%%%%%%%%%%%%%%%%%%%%%%%%%%%%
%%%%%%%%%%%%%%%%%%%%%%%%%%%%%%%%%%%%%%%%%%%%%%%%%%%%%%%%%%%%%
%
%
%
%
%
%
%
%%%%%%%%%%%%%%%%%%%%%%%%%%%%%%%%%%%%%%%%%%%%%%%%%%%%%%%%%%%%%%%%%%%%%%%%%%
%%%% \BEGIN BIBLIOGRAPHY
%%%%%%%%%%%%%%%%%%%%%%%%%%%%%%%%%%%%%%%%%%%%%%%%%%%%%%%%%%%%%%%%%%%%%%%%%%
\bibliographystyle{alpha}%wmaainf}%alpha}%plain}

%\bibliographystyle{abbrvnat} % Choose Phys. Rev. style for bibliography
%\bibliographystyle{abbrv}

%\newpage
\bibliography{BibFileGoesHere}  % ``name``.bib is the name of thedatabase

\begin{bibdiv}
\begin{biblist}

\bib{adams2021entropic}{misc}{
      author={Adams, Daniel},
      author={Duong, Manh~Hong},
      author={dos Reis, Goncalo},
       title={Entropic regularisation of non-gradient systems},
        date={2021},
        note={arXiv:2104.04372},
}

\bib{adams2011large}{article}{
      author={Adams, Stefan},
      author={Dirr, Nicolas},
      author={Peletier, Mark~A.},
      author={Zimmer, Johannes},
       title={From a large-deviations principle to the {W}asserstein gradient
  flow: a new micro-macro passage},
        date={2011},
        ISSN={0010-3616},
     journal={Comm. Math. Phys.},
      volume={307},
      number={3},
       pages={791\ndash 815},
         url={https://doi.org/10.1007/s00220-011-1328-4},
      review={\MR{2842966}},
}

\bib{arnold2012wigner}{article}{
      author={Arnold, Anton},
      author={Gamba, Irene~M.},
      author={Gualdani, Maria~Pia},
      author={Mischler, St\'{e}phane},
      author={Mouhot, Clement},
      author={Sparber, Christof},
       title={The {W}igner-{F}okker-{P}lanck equation: stationary states and
  large time behavior},
        date={2012},
        ISSN={0218-2025},
     journal={Math. Models Methods Appl. Sci.},
      volume={22},
      number={11},
       pages={1250034, 31},
         url={https://doi.org/10.1142/S0218202512500340},
      review={\MR{2974172}},
}

\bib{ambrosio2008gradient}{book}{
      author={Ambrosio, Luigi},
      author={Gigli, Nicola},
      author={Savar{\'e}, Giuseppe},
       title={Gradient flows: in metric spaces and in the space of probability
  measures},
   publisher={Springer Science \& Business Media},
        date={2008},
}

\bib{bowles2015weak}{article}{
      author={Bowles, Malcolm},
      author={Agueh, Martial},
       title={Weak solutions to a fractional {F}okker-{P}lanck equation via
  splitting and {W}asserstein gradient flow},
        date={2015},
        ISSN={0893-9659},
     journal={Appl. Math. Lett.},
      volume={42},
       pages={30\ndash 35},
         url={https://doi.org/10.1016/j.aml.2014.10.008},
      review={\MR{3294365}},
}

\bib{bernton2018langevin}{inproceedings}{
      author={Bernton, Espen},
       title={Langevin {M}onte {C}arlo and {JKO} splitting},
organization={PMLR},
        date={2018},
   booktitle={Conference on learning theory},
      volume={75},
       pages={1777\ndash 1798},
}

\bib{bodineau2008large}{article}{
      author={Bodineau, Thierry},
      author={Lefevere, Rapha\"{e}l},
       title={Large deviations of lattice {H}amiltonian dynamics coupled to
  stochastic thermostats},
        date={2008},
        ISSN={0022-4715},
     journal={J. Stat. Phys.},
      volume={133},
      number={1},
       pages={1\ndash 27},
         url={https://doi.org/10.1007/s10955-008-9601-4},
      review={\MR{2438895}},
}

\bib{salamontheo}{book}{
      author={B\"{u}hler, Theo},
      author={Salamon, Dietmar~A.},
       title={Functional analysis},
      series={Graduate Studies in Mathematics},
   publisher={American Mathematical Society, Providence, RI},
        date={2018},
      volume={191},
        ISBN={978-1-4704-4190-6},
         url={https://doi.org/10.1090/gsm/191},
      review={\MR{3823238}},
}

\bib{carlier2017convergence}{article}{
      author={Carlier, Guillaume},
      author={Duval, Vincent},
      author={Peyr{\'e}, Gabriel},
      author={Schmitzer, Bernhard},
       title={Convergence of entropic schemes for optimal transport and
  gradient flows},
        date={2017},
     journal={SIAM Journal on Mathematical Analysis},
      volume={49},
      number={2},
       pages={1385\ndash 1418},
}

\bib{carlen2004solution}{article}{
      author={Carlen, E.~A.},
      author={Gangbo, W.},
       title={Solution of a model {B}oltzmann equation via steepest descent in
  the 2-{W}asserstein metric},
        date={2004},
        ISSN={0003-9527},
     journal={Arch. Ration. Mech. Anal.},
      volume={172},
      number={1},
       pages={21\ndash 64},
         url={https://doi.org/10.1007/s00205-003-0296-z},
      review={\MR{2048566}},
}

\bib{carlier2017splitting}{article}{
      author={Carlier, Guillaume},
      author={Laborde, Maxime},
       title={A splitting method for nonlinear diffusions with nonlocal,
  nonpotential drifts},
        date={2017},
     journal={Nonlinear Analysis: Theory, Methods \& Applications},
      volume={150},
       pages={1\ndash 18},
}

\bib{Carlen2017}{article}{
      author={Carlen, E.~A.},
      author={Maas, J.},
       title={Gradient flow and entropy inequalities for quantum {M}arkov
  semigroups with detailed balance},
        date={2017},
        ISSN={0022-1236},
     journal={Journal of Functional Analysis},
      volume={273},
      number={5},
       pages={1810\ndash 1869},
  url={https://www.sciencedirect.com/science/article/pii/S0022123617301878},
}

\bib{carrillo1995initial}{article}{
      author={Carrillo, Jos\'{e}~A.},
      author={Soler, Juan},
       title={On the initial value problem for the
  {V}lasov-{P}oisson-{F}okker-{P}lanck system with initial data in {$L^p$}
  spaces},
        date={1995},
        ISSN={0170-4214},
     journal={Math. Methods Appl. Sci.},
      volume={18},
      number={10},
       pages={825\ndash 839},
         url={https://doi.org/10.1002/mma.1670181006},
      review={\MR{1343393}},
}

\bib{cuturi2013sinkhorn}{article}{
      author={Cuturi, Marco},
       title={Sinkhorn distances: Lightspeed computation of optimal transport},
        date={2013},
     journal={Advances in neural information processing systems},
      volume={26},
       pages={2292\ndash 2300},
}

\bib{DuongLu2018}{article}{
      author={Duong, M.~H.},
      author={Lu, Y.},
       title={An operator splitting scheme for the fractional kinetic
  {F}okker-{P}lanck equation},
        date={2019},
     journal={Discrete \& Continuous Dynamical Systems},
      volume={39},
      number={10},
       pages={5707\ndash 5727},
}

\bib{diperna1989ordinary}{article}{
      author={DiPerna, R.~J.},
      author={Lions, P.-L.},
       title={Ordinary differential equations, transport theory and {S}obolev
  spaces},
        date={1989},
        ISSN={0020-9910},
     journal={Invent. Math.},
      volume={98},
      number={3},
       pages={511\ndash 547},
         url={https://doi.org/10.1007/BF01393835},
      review={\MR{1022305}},
}

\bib{delarue2010density}{article}{
      author={Delarue, Fran{\c{c}}ois},
      author={Menozzi, St{\'e}phane},
       title={Density estimates for a random noise propagating through a chain
  of differential equations},
        date={2010},
     journal={Journal of functional analysis},
      volume={259},
      number={6},
       pages={1577\ndash 1630},
}

\bib{DiMario2020}{article}{
      author={Di~Marino, S.},
      author={Chizat, L.},
       title={A tumor growth model of {H}ele-{S}haw type as a gradient flow},
        date={2020},
     journal={ESAIM: COCV},
      volume={26},
       pages={103},
}

\bib{Duong2013}{article}{
      author={Duong, Manh~Hong},
      author={Peletier, Mark~A.},
      author={Zimmer, Johannes},
       title={G{ENERIC} formalism of a {V}lasov-{F}okker-{P}lanck equation and
  connection to large-deviation principles},
        date={2013},
        ISSN={0951-7715},
     journal={Nonlinearity},
      volume={26},
      number={11},
       pages={2951\ndash 2971},
         url={https://doi.org/10.1088/0951-7715/26/11/2951},
      review={\MR{3129075}},
}

\bib{duong2014conservative}{article}{
      author={Duong, M.~H.},
      author={Peletier, M.~A.},
      author={Zimmer, J.},
       title={Conservative-dissipative approximation schemes for a generalized
  {K}ramers equation},
        date={2014},
     journal={Math. Methods Appl. Sci.},
      volume={37},
      number={16},
       pages={2517\ndash 2540},
}

\bib{DuongTran17}{article}{
      author={Duong, Manh~Hong},
      author={Tran, Hoang~Minh},
       title={Analysis of the mean squared derivative cost function},
        date={2017},
     journal={Mathematical Methods in the Applied Sciences},
      volume={40},
      number={14},
       pages={5222\ndash 5240},
}

\bib{DuongTran18}{article}{
      author={Duong, Manh~Hong},
      author={Tran, Hoang~Minh},
       title={On the fundamental solution and a variational formulation for a
  degenerate diffusion of {K}olmogorov type},
        date={2018},
        ISSN={1078-0947},
     journal={Discrete Contin. Dyn. Syst.},
      volume={38},
      number={7},
       pages={3407\ndash 3438},
         url={https://doi.org/10.3934/dcds.2018146},
      review={\MR{3809088}},
}

\bib{Duong2015NonlinearAnalysis}{article}{
      author={Duong, Manh~Hong},
       title={Long time behaviour and particle approximation of a generalised
  {V}lasov dynamic},
        date={2015},
        ISSN={0362-546X},
     journal={Nonlinear Anal.},
      volume={127},
       pages={1\ndash 16},
         url={https://doi.org/10.1016/j.na.2015.06.018},
      review={\MR{3392354}},
}

\bib{Esposito2021}{article}{
      author={Esposito, A.},
      author={Patacchini, F.~S.},
      author={Schlichting, A.},
      author={Slep{\v{c}}ev, D.},
       title={Nonlocal-interaction equation on graphs: Gradient flow structure
  and continuum limit},
        date={2021May},
        ISSN={1432-0673},
     journal={Archive for Rational Mechanics and Analysis},
      volume={240},
      number={2},
       pages={699\ndash 760},
         url={https://doi.org/10.1007/s00205-021-01631-w},
}

\bib{GlowinskiMR3587821SplitMethodsBook}{book}{
      editor={Glowinski, Roland},
      editor={Osher, Stanley~J.},
       title={Splitting methods in communication, imaging, science, and
  engineering},
      series={Scientific Computation},
   publisher={Springer, Cham},
        date={2016},
        ISBN={978-3-319-41587-1; 978-3-319-41589-5},
         url={https://doi.org/10.1007/978-3-319-41589-5},
      review={\MR{3587821}},
}

\bib{huang2000variational}{article}{
      author={Huang, Chaocheng},
      author={Jordan, Richard},
       title={Variational formulations for {V}lasov-{P}oisson-{F}okker-{P}lanck
  systems},
        date={2000},
        ISSN={0170-4214},
     journal={Math. Methods Appl. Sci.},
      volume={23},
      number={9},
       pages={803\ndash 843},
  url={https://doi.org/10.1002/1099-1476(200006)23:9<803::AID-MMA137>3.3.CO;2-R},
      review={\MR{1763126}},
}

\bib{hwang2011vlasov}{article}{
      author={Hwang, Hyung~Ju},
      author={Jang, Juhi},
       title={On the {V}lasov-{P}oisson-{F}okker-{P}lanck equation near
  {M}axwellian},
        date={2013},
        ISSN={1531-3492},
     journal={Discrete Contin. Dyn. Syst. Ser. B},
      volume={18},
      number={3},
       pages={681\ndash 691},
         url={https://doi.org/10.3934/dcdsb.2013.18.681},
      review={\MR{3007749}},
}

\bib{Huang00}{article}{
      author={Huang, Chaocheng},
       title={A variational principle for the {K}ramers equation with unbounded
  external forces},
        date={2000},
        ISSN={0022-247X},
     journal={J. Math. Anal. Appl.},
      volume={250},
      number={1},
       pages={333\ndash 367},
         url={https://doi.org/10.1006/jmaa.2000.7109},
      review={\MR{1893894}},
}

\bib{jordan1998variational}{article}{
      author={Jordan, Richard},
      author={Kinderlehrer, David},
      author={Otto, Felix},
       title={The variational formulation of the {F}okker-{P}lanck equation},
        date={1998},
        ISSN={0036-1410},
     journal={SIAM J. Math. Anal.},
      volume={29},
      number={1},
       pages={1\ndash 17},
         url={https://doi.org/10.1137/S0036141096303359},
      review={\MR{1617171}},
}

\bib{Jabin2018}{article}{
      author={Jabin, P.-E.},
      author={Wang, Z.},
       title={Quantitative estimates of propagation of chaos for stochastic
  systems with ${W}^{-1,\infty}$ kernels},
        date={2018Oct},
     journal={Inventiones mathematicae},
      volume={214},
      number={1},
       pages={523\ndash 591},
}

\bib{Kramers40}{article}{
      author={Kramers, H.~A.},
       title={Brownian motion in a field of force and the diffusion model of
  chemical reactions},
        date={1940},
     journal={Physica},
      volume={7},
       pages={284\ndash 304},
      review={\MR{0002962 (2,140d)}},
}

\bib{KnoppSinkhorn1967}{article}{
      author={Knopp, P.},
      author={Sinkhorn, R.},
       title={{Concerning nonnegative matrices and doubly stochastic
  matrices.}},
        date={1967},
     journal={Pacific Journal of Mathematics},
      volume={21},
      number={2},
       pages={343 \ndash  348},
         url={https://doi.org/},
}

\bib{kupferman2004fractional}{article}{
      author={Kupferman, Raz},
       title={Fractional kinetics in {K}ac-{Z}wanzig heat bath models},
        date={2004},
        ISSN={0022-4715},
     journal={J. Statist. Phys.},
      volume={114},
      number={1-2},
       pages={291\ndash 326},
         url={https://doi.org/10.1023/B:JOSS.0000003113.22621.f0},
      review={\MR{2032133}},
}

\bib{laborde2016some}{incollection}{
      author={Laborde, Maxime},
       title={On some nonlinear evolution systems which are perturbations of
  {W}asserstein gradient flows},
        date={2017},
   booktitle={Topological optimization and optimal transport},
      series={Radon Ser. Comput. Appl. Math.},
      volume={17},
   publisher={De Gruyter, Berlin},
       pages={304\ndash 332},
      review={\MR{3729381}},
}

\bib{lisini2009nonlinear}{article}{
      author={Lisini, Stefano},
       title={Nonlinear diffusion equations with variable coefficients as
  gradient flows in {W}asserstein spaces},
        date={2009},
        ISSN={1292-8119},
     journal={ESAIM Control Optim. Calc. Var.},
      volume={15},
      number={3},
       pages={712\ndash 740},
         url={https://doi.org/10.1051/cocv:2008044},
      review={\MR{2542579}},
}

\bib{Liu2021}{article}{
      author={Liu, C.},
      author={Wang, C.},
      author={Wang, Y.},
       title={A structure-preserving, operator splitting scheme for
  reaction-diffusion equations with detailed balance},
        date={2021},
        ISSN={0021-9991},
     journal={Journal of Computational Physics},
      volume={436},
       pages={110253},
  url={https://www.sciencedirect.com/science/article/pii/S0021999121001480},
}

\bib{Maas2011}{article}{
      author={Maas, Jan},
       title={Gradient flows of the entropy for finite {M}arkov chains},
        date={2011},
        ISSN={0022-1236},
     journal={Journal of Functional Analysis},
      volume={261},
      number={8},
       pages={2250\ndash 2292},
  url={https://www.sciencedirect.com/science/article/pii/S0022123611002278},
}

\bib{Mielke2014}{article}{
      author={Mielke, A.},
      author={Peletier, M.~A.},
      author={Renger, D. R.~M.},
       title={On the relation between gradient flows and the large-deviation
  principle, with applications to {M}arkov chains and diffusion},
        date={2014Nov},
        ISSN={1572-929X},
     journal={Potential Analysis},
      volume={41},
      number={4},
       pages={1293\ndash 1327},
         url={https://doi.org/10.1007/s11118-014-9418-5},
}

\bib{markowich2012semiconductor}{book}{
      author={Markowich, P.~A.},
      author={Ringhofer, C.~A.},
      author={Schmeiser, C.},
       title={Semiconductor equations},
   publisher={Springer-Verlag, Vienna},
        date={1990},
        ISBN={3-211-82157-0},
         url={https://doi.org/10.1007/978-3-7091-6961-2},
      review={\MR{1063852}},
}

\bib{marcos2020solutions}{article}{
      author={Marcos, Aboubacar},
      author={Soglo, Ambroise},
       title={Solutions of a class of degenerate kinetic equations using
  steepest descent in {W}asserstein space},
        date={2020},
     journal={Journal of Mathematics},
      volume={2020},
}

\bib{matthes2020discretization}{article}{
      author={Matthes, Daniel},
      author={S{\"o}llner, Benjamin},
       title={Discretization of flux-limited gradient flows:
  $\gamma$-convergence and numerical schemes},
        date={2020},
     journal={Mathematics of Computation},
      volume={89},
      number={323},
       pages={1027\ndash 1057},
}

\bib{ottobre2011asymptotic}{article}{
      author={Ottobre, M.},
      author={Pavliotis, G.~A.},
       title={Asymptotic analysis for the generalized {L}angevin equation},
        date={2011},
        ISSN={0951-7715},
     journal={Nonlinearity},
      volume={24},
      number={5},
       pages={1629\ndash 1653},
         url={https://doi.org/10.1088/0951-7715/24/5/013},
      review={\MR{2793823}},
}

\bib{ottinger2018generic}{article}{
      author={{\"O}ttinger, Hans~Christian},
       title={{GENERIC}: Review of successful applications and a challenge for
  the future},
        date={2018},
     journal={arXiv preprint arXiv:1810.08470},
}

\bib{peyre2019computational}{article}{
      author={Peyr{\'e}, Gabriel},
      author={Cuturi, Marco},
       title={Computational optimal transport: with applications to data
  science},
        date={2019},
     journal={Foundations and Trends{\textregistered} in Machine Learning},
      volume={11},
      number={5-6},
       pages={355\ndash 607},
}

\bib{peyre2015entropic}{article}{
      author={Peyr{\'e}, Gabriel},
       title={Entropic approximation of wasserstein gradient flows},
        date={2015},
     journal={SIAM Journal on Imaging Sciences},
      volume={8},
      number={4},
       pages={2323\ndash 2351},
}

\bib{rey2006open}{incollection}{
      author={Rey-Bellet, Luc},
       title={Open classical systems},
        date={2006},
   booktitle={Open quantum systems. {II}},
      series={Lecture Notes in Math.},
      volume={1881},
   publisher={Springer, Berlin},
       pages={41\ndash 78},
         url={https://doi.org/10.1007/3-540-33966-3_2},
      review={\MR{2248987}},
}

\bib{risken1989fokker}{misc}{
      author={Risken, H.},
       title={The {F}okker-{P}lanck equation},
     edition={Second},
      series={Springer Series in Synergetics},
   publisher={Springer-Verlag, Berlin},
        date={1989},
      volume={18},
        ISBN={3-540-50498-2},
         url={https://doi.org/10.1007/978-3-642-61544-3},
        note={Methods of solution and applications},
      review={\MR{987631}},
}

\bib{santambrogio2015optimal}{article}{
      author={Santambrogio, Filippo},
       title={Optimal transport for applied mathematicians},
        date={2015},
     journal={Birk{\"a}user, NY},
      volume={55},
      number={58-63},
       pages={94},
}

\bib{serfaty2020mean}{article}{
      author={Serfaty, Sylvia},
       title={Mean field limit for {C}oulomb-type flows},
        date={2020},
        ISSN={0012-7094},
     journal={Duke Math. J.},
      volume={169},
      number={15},
       pages={2887\ndash 2935},
         url={https://doi.org/10.1215/00127094-2020-0019},
        note={With an appendix by Mitia Duerinckx and Serfaty},
      review={\MR{4158670}},
}

\bib{villani2008optimal}{book}{
      author={Villani, C{\'e}dric},
       title={Optimal transport: old and new},
   publisher={Springer Science \& Business Media},
        date={2008},
      volume={338},
}

\bib{dric2009hypocoercivity}{book}{
      author={Villani, C{\'e}dric},
       title={Hypocoercivity},
   publisher={American Mathematical Soc.},
        date={2009},
      volume={202},
}

\bib{villani2021topics}{book}{
      author={Villani, C{\'e}dric},
       title={Topics in optimal transportation},
   publisher={American Mathematical Soc.},
        date={2021},
      volume={58},
}

\bib{yao2013blow}{article}{
      author={Yao, Yao},
      author={Bertozzi, Andrea~L},
       title={Blow-up dynamics for the aggregation equation with degenerate
  diffusion},
        date={2013},
     journal={Physica D: Nonlinear Phenomena},
      volume={260},
       pages={77\ndash 89},
}

\end{biblist}
\end{bibdiv}


%%%%%%%%%%%%%%%%%%%%%%%%%%%%%%%%%%%%%%%%%%%%%%%%%%%%%%%%%%%%%%%%%%%%%%%%%%
%%%% \END BIBLIOGRAPHY
%%%%%%%%%%%%%%%%%%%%%%%%%%%%%%%%%%%%%%%%%%%%%%%%%%%%%%%%%%%%%%%%%%%%%%%%%%

\end{document}